\documentclass[12pt]{amsart}
\usepackage{amsmath, amssymb,amsthm,tikz}
\usepackage{mathrsfs}
\usepackage{mathtools}
\usepackage{hyperref}
\usepackage{multirow}

\usepackage{hyperref}
\hypersetup{
    colorlinks,
    citecolor=blue,
    filecolor=blue,
    linkcolor=blue,
    urlcolor=black
}

\numberwithin{equation}{subsection}
\newtheorem{cor}[equation]{Corollary}
\newtheorem{defn}[equation]{Definition}
\newtheorem{notation}[equation]{Notation}
\newtheorem{remark}[equation]{Remark}
\newtheorem{prop}[equation]{Proposition}
\newtheorem{example}[equation]{Example}
\newtheorem{examples}[equation]{Examples}

\newtheorem{thm}[equation]{Theorem}

\usepackage[margin=1.19in]{geometry}

\newcommand{\NN}{\mathbb{N}}
\newcommand{\ZZ}{\mathbb{Z}}

\newcommand{\CC}{\mathbb{C}}

\newcommand{\Qlbar}{\overline{\mathbb{Q}}_\ell}
\newcommand{\FF}{\mathbb{F}}

\newcommand{\Fq}{{\mathbb{F}_q}}

\newcommand{\G}{\mathbb{G}}

\newcommand{\cE}{\mathcal{E}}

\newcommand{\cL}{\mathcal{L}}
\newcommand{\tcL}{{\widetilde{\cL}}}

\newcommand{\cP}{\mathcal{P}}
\newcommand{\cS}{\mathcal{S}}
\newcommand{\cT}{\mathcal{T}}

\newcommand{\End}{\operatorname{End}}
\newcommand{\sgn}{\operatorname{sgn}}

\newcommand{\Id}{\operatorname{Id}}
\newcommand{\sign}{\operatorname{sign}}
\newcommand{\rI}{\mathrm{I}}

\newcommand{\mult}{\mathrm{m}}

\newcommand{\Ind}{\operatorname{Ind}}

\newcommand{\rU}{\mathrm{U}}

\newcommand{\bB}{\mathbf{B}}
\newcommand{\bG}{\mathbf{G}}
\newcommand{\bL}{\mathbf{L}}

\newcommand{\bP}{\mathbf{P}}

\newcommand{\bT}{\mathbf{T}}
\newcommand{\bU}{\mathbf{U}}

\newcommand{\an}{\operatorname{an}}

\newcommand{\disc}{\operatorname{disc}}

\newcommand{\unip}{\mathrm{u}}

\newcommand{\ord}{\operatorname{ord}}

\newcommand{\height}{\operatorname{height}}

\newcommand{\rank}{\operatorname{rank}}
\newcommand{\rk}{\operatorname{rk}}

\newcommand{\Tr}{\operatorname{Tr}}

\newcommand{\Hom}{\operatorname{Hom}}

\newcommand{\Res}{\operatorname{Res}}
\newcommand{\rr}{\operatorname{r}}
\newcommand{\ii}{\operatorname{i}}

\newcommand{\Cent}{\operatorname{C}}

\newcommand{\Nor}{\operatorname{N}}
\newcommand{\fR}{\mathfrak{R}}
\newcommand{\fs}{\mathfrak{s}}

\newcommand{\fC}{\mathfrak{C}}

\newcommand{\Irr}{\operatorname{Irr}}

\newcommand{\fS}{\mathfrak{S}}

\newcommand{\rN}{{\mathrm{N}}}

\newcommand{\rZ}{{\mathrm{Z}}}

\newcommand{\Gal}{\operatorname{Gal}}

\newcommand{\GL}{\mathrm{GL}}

\newcommand{\Sp}{\operatorname{Sp}}
\newcommand{\SO}{\operatorname{SO}}

\newcommand{\Mat}{\operatorname{M}}

\newcommand{\cusp}{\mathrm{cusp}}

\newcommand{\Sc}{\mathrm{Sc}}

\newcommand{\bbH}{\mathbb{H}}

\newcommand{\cc}{\mathrm{c}}

\def\Sym{{\mathrm{Sym}}}
\def\infl{\mathrm{infl}}

\def\bX{{\mathbf{X}}}

\def\ord{\mathrm{ord}}

\def\disc{{\mathrm{disc}}}

\def\rO{{\mathrm{O}}}

\def\rB{{\mathrm{B}}}

\def\rD{{\mathrm{D}}}

\def\cc{\mathrm{sc}}

\def\an{{\mathrm{an}}}
\def\rr{{\mathrm{r}}}

\def\cc{{\mathrm{c}}}
\def\cB{{\mathcal{B}}}

\def\cT{{\mathcal{T}}}

\renewcommand{\tilde}{\widetilde}

\def\defect{{\mathrm{def}}}
\def\deter{{\mathrm{det}}}
\def\disc{{\mathrm{disc}}}
\def\modu{{\mathrm{mod}}}

\def\modu{{\mathrm{mod}}}
\def\prin{{\mathrm{prin}}}

\title[The theta correspondence over finite fields]{The theta correspondence over finite fields}

\author{Anne-Marie Aubert}

\address{Sorbonne Universit\'e and Universit\'e Paris Cit\'e, CNRS,
IMJ-PRG, F-75005 Paris, France}
\email{anne-marie.aubert@imj-prg.fr}
\date{\today}
\begin{document}
\maketitle
\begin{abstract} This set of lecture notes is an expanded version of a mini-course the author gave in March of 2025 for the program 
``\href{https://indico.math.cnrs.fr/event/10843/}{Representation Theory \& Noncommutative Geometry}" at the Institut Henri Poincaré, Paris.  

The goal is to provide a survey of the main properties of the theta correspondence over finite fields of odd characteristic, including its compatibility with Harish-Chandra and Lusztig series, and with the Jordan decomposition of representations,
as well as its full explicit description. 
\end{abstract}

\vspace*{12pt}

\tableofcontents

\section{Introduction} \label{sec:intro}

In 1964, André Weil introduced the \textit{Weil representation}  (also called \textit{oscillator representation}), initially for adelic and local fields, see \cite{Weil}. 
Weil showed that this representation encodes deep symmetries linking symplectic and orthogonal structures, and
also observed that the same construction makes sense over finite fields, where everything becomes concrete and combinatorial.

Let $W$ be a finite dimensional vector space over a finite field $k$ of odd characteristic, let $\langle\;,\;\rangle$ be a nondegenerate alternated form on $V$, and let $\Sp(W)$  denote the  symplectic group attached to 
$\langle\;,\;\rangle$. The  Weil representation is a representation of the group $\Sp(W)$, the construction of which relies on the Stone-von Neumann Theorem for the Heisenberg group associated to the vector space $W$. 
In  the local fields setting the Weil representation is only a projective representation, which makes necessary to consider a metaplectic  cover of the symplectic group (see \cite{Weil} or \cite{AP}). The finite fields setting (in the odd characteristic case) is hence simpler in this aspect. 

In the 1970s, Roger Howe introduced in \cite{Howe-finite,How} the concept of \textit{dual pairs} (these are the pairs $(G,G')$ of reductive subgroups of  $\Sp(W)$ such that $G'$ is the  centralizer of $G$ in $\Sp(W)$ and vice-versa) 
and defined a correspondence between the categories of complex representations of these subgroups. 

This correspondence, which is known as the \textit{theta correspondence} or the \textit{Howe correspondence}, is obtained from the restriction of the Weil representation to $G\cdot G'$.
The image under the theta correspondence for a dual pair $(G,G')$ of a given irreducible representation $\pi$ of $G$ is a finite set, say $\{\pi'_1,\ldots,\pi'_{r'}\}$ of irreducible representations $\pi'_i$ of $G'$, which in general,  in contrast with the  the local fields setting,  is not a singleton.  

The explicit description of the theta correspondence over finite fields is a long standing question, which is the main subject of these notes. It has also applications to the theta correspondence over nonarchimedean local fields as shown notably in \cite{Pan0, LM, Au}.  

\smallskip

These notes are organized as follows. Section~\ref{sec:Weil} introduces the Weil representation of $\Sp(W)$, and section~\ref{sec:theta} describes the construction of the theta correspondence and some of its basic properties.

Section~\ref{sec:thetaHC} concerns the ``philosophy of cusp forms" of Harish-Chandra and its application to the theta correspondence. More precisely, let $\bG$ be an arbitrary reductive connected group or a split orthogonal group defined over $k$, and $G$ denote its group of $k$-points. The group $G$ is a finite group of Lie type. 
The category  $\fR(G)$ of  complex representations of $G$ decomposes as a direct product of subcategories $\fR^\fs(G)$, where $\fs$ is the $G$-conjugacy class of a \textit{cuspidal pair} $(L,\sigma)$, that is, $L$ is a Levi subgroup of $G$ and $\sigma$ an irreducible cuspidal representation of $L$. We describe the set $\Irr^\fs(G)$ of isomorphism classes of irreducible objects of $\fR^\fs(G)$, called the \textit{Harish-Chandra series} associated to $\fs$, and, in Theorem~\ref{thm:HC-HC} (proved in \cite{AMR}), we formulate the compatibility of the theta correspondence with the Harish-Chandra series of $G$ and $G'$.

In section~\ref{sec:Lusztig}, we focus on Lusztig's approach of the study of representations of finite groups of Lie type.  Let $\check\bG^{\circ}$ denote the reductive connected group with root datum dual to that of the identity component $\bG^\circ$ of $\bG$. We consider the  partition of the set $\Irr(G)$ of isomorphism classes of irreducible complex representations of $G$ in \textit{Lusztig series} $\cE(G,s)$, where $(s)$ runs over  the $\check G^{\circ}$-conjugacy classes of semisimple elements of $\check G^{\circ}$.  
Let $\bG_s$ be the \textit{endoscopic group} of $\bG$ associated with $s$, that is,  the dual group of the centralizer of $s$ in $\check \bG$. By Lusztig's results  \cite{Lubook,Lus-disconnected-center} (and \cite{AMR} for $\bG=\rO^\pm_{2n}$),  the set $\mathcal{E}(G,s)$ is in bijection with the \textit{unipotent Lusztig series} $\mathcal{E}(G_s,1)$ of $G_s$. We show that the description of the theta correspondence reduces to describe its restriction to unipotent representations.

Section~\ref{sec:unip} describes the irreducible constituents of unipotent Harish-Chandra series for every reductive dual pair $(G,G')$ and expresses the theta correspondence for unipotent representations 
as a correspondence between representations of two finite Weyl groups. These latter was proved for the dual pair of general linear groups and for the dual of unitary groups in \cite{AMR}, and conjectured in \cite{AMR} for ortho-symplectic dual pairs.
The validity of our conjecture has been established independently by S.-Y.~Pan in \cite{PanAJM} and by J.-j. Ma, C. Qiu, and J. Zou in \cite{MQZ}.

The main goal of section~\ref{sec:combi} is to reformulate, following Pan's works, the correspondence between representations of Weyl groups in terms of  Lusztig's combinatorics of symbols. 
Section~\ref{sec:one-to-one} describes several manners to extract a one-to-one correspondence from the theta correspondence for arbitrary dual pairs and representations.

Finally, section~\ref{sec:holo} briefly relates the theta correspondence to the holographic correspondence studies by A. Dymarsky, J. Henriksson and B. McPeak in \cite{DHM}.

\bigskip

\noindent
\textbf{Acknowledgements.} I wish to warmly thank the Institut Henri Poincar\'e for their hospitality during the 2025 Thematic Trimester ``\href{https://indico.math.cnrs.fr/event/10843/}{Representation Theory \& Noncommutative Geometry}". It is a pleasure to thank  
Hung Yean Loke and Tomasz Przebinda, the organizers of the ``Theta Week", for having invited me to give these lectures, and all of the participants for a  very stimulating and pleasing meeting.
 
\section{The Weil representation} \label{sec:Weil}
Let $k=\Fq$ be a finite field with $q$ elements and let  $\overline{k}$ be a fixed algebraic closure of $k$. We suppose $q$ odd. Let $N$ be a positive integer,
Let $\Mat_{2N}(k)$ denote the ring of $2N\times 2N$ matrices with entries in $k$. For  $g\in\Mat_{2N}(k)$, we denote by $g^t$ its transposed.
 
Let $W$ be a $k$-vector space of dimension $2N$, equipped with a non-degenerate alternated form  $\langle\;,\;\rangle$. 
The group of isometries of  $\langle\;,\;\rangle$  is the \textit{symplectic group}:
\begin{equation}
\Sp(W) := \left\{g \in\GL(W) \;:\; \langle g(w_1),g(w_2)\rangle=  \langle w_1,w_2  \rangle\;\text{ for all $w_1,w_2\in W$}\right\}.
\end{equation}

The elements of $\Sp(W)=\Sp_{2n}(k)$ admit the following characterization:
Let $g=\left(\begin{smallmatrix} a& b\cr c& d\end{smallmatrix}\right)\in\GL_{2N}(\Fq)$, with $a,b,c,d\in\Mat_{2N}(k)$. Then 
\begin{align*}
g\in \Sp_{2n}(k)&\;\;\Leftrightarrow\;\;&\text{$a^t c=c^t a$, $b^t d=d^t b$ and $a^t d-c^t b=\rI$}\cr
&\;\;\Leftrightarrow\;\;&\text{$ab^t=ba^t$, $cd^t=dc^t$ and $ad^t -bc^t=\rI$}
\end{align*}

\begin{defn} \label{defn:Heis}
{\rm The Heisenberg group $H(W)$:
As a set, we have
\[H(W):=W\times k=k^{2N+1}.\] 
The multiplication law is defined by
\[(w,t)\cdot (w',t'):=(w+w',t+t'+\frac{1}{2}\langle w,w'\rangle).\]}
\end{defn}

\begin{defn}  {\rm Let $H$ be a finite group. 
\begin{itemize}
\item
The \textit{exponent} of  $H$  is the least positive $n$ s.t. $h^n=1$ $\forall h\in H$.
\item
The \textit{Frattini subgroup} of $H$, denoted by $\Phi(H)$, is  the intersection proper maximal subgroups of $H$.
\item
We denote by $\rZ_H$ the center of $H$ and by $H'$ the derived subgroup of $H$.  
If $H$  is a $p$-group, then $H$ is \textit{special} if either $H$ is elementary abelian or $H$ is of class $2$ (i.e., $H' \subset\rZ_H$), and $H'= \rZ_H = \Phi(H)$ is elementary
abelian. 
\item
The group $H$ is  \textit{extraspecial} if it is special and $|H'|=p$.
\end{itemize} }
\end{defn}

Let  $\rZ_{H(W)}$ denote the center of $H(W)$. It  is isomorphic to $k$.
\begin{remark} {\rm $H=H(W)$ is a special $p$-group with exponent $p$, with derived subgroup $H'=(k,+)$. If $q=p$, with $p$ a prime number, $H$ is an 
extraspecial $p$-group.}
\end{remark}

\begin{thm} \label{thm:StonevN} {\rm [Finite analogue of the Stone-von-Neumann Theorem]} 
For every non-trivial character $\psi$ of $\rZ_{H(W)}$,  there exists a unique  (up to equivalence) irreducible representation $(\varrho_\psi,S)$ of $H(W)$,  known as the \textit{Heisenberg representation}, such that 
\[\varrho_\psi(z)=\psi(z)\,\Id_S\;\;\text{for all $z\in\rZ_{H(W)}$.}\]
\end{thm}
\begin{proof}
See  \cite[Theorem~2.I.2]{MVW}.
\end{proof}

The group $\Sp(W)$ acts as a group of automorphisms of $H(W)$ by the rule:
\begin{equation}
g\cdot (w,t):=(g(w),t)\quad\text{for $g\in \Sp(W)$ and $(w,t)\in H(W)$}.
\end{equation}
Let $g\in \Sp(W)$. The action of $\Sp(W)$ on $H(W)$ fixes the elements of its center. Hence, for a fixed character $\psi$ of $k$, the representations $\varrho_\psi$ and $g\cdot \varrho_\psi$ agree on $\rZ_{H(W)}$, for any $g\in\Sp(W)$.  
Theorem~\ref{thm:StonevN} implies that there is an operator $\omega_\psi(g)$ verifying
\begin{equation} \label{eqn:Weil}
\varrho_\psi(g\cdot w,t) = \omega_\psi(g)\varrho_\psi(w,t)\omega_\psi(g)^{-1}.
\end{equation}
This defines a projective representation $\omega_\psi$ of $\Sp(W)$, which can be lifted to an actual representation of $\Sp(W)$, still denoted by $\omega_\psi$, since $H^2(\Sp(W),\CC^\times)=0$.

\begin{defn} \label{defn:Weil}
The representation $\omega_\psi$ of $\Sp(W)$ is the \textit{Weil representation} (or \textit{oscillator representation}) determined by the additive character $\psi$.
\end{defn}

\begin{remark} {\rm 
For $a\in k$, let $\psi_a\colon \rZ_{H(W)}\to\CC$ be the character  defined by $\psi_a(x):=\psi(ax)$. 
\begin{itemize}
\item 
For $b\in k^\times$, the representations $\omega_{\psi_a}$ and $\omega_{\psi_{ab^2}}$ are isomorphic. 
\item
The representations $\omega_{\psi_a}$, as $a$ runs over representatives for $k^\times/(k^\times)^2$, are not isomorphic.
\end{itemize}
}
\end {remark}

\section{The theta correspondence} \label{sec:theta}

\subsection{Reductive dual pairs}\

\begin{defn} {\rm
A pair of reductive subgroups of $\Sp(W)$,  is called  a \textit{dual pair} if 
\[\Cent_{\Sp(W)}(G)=G'\quad\text{and}\quad \Cent_{\Sp(W)}(G')=G.\]
}
\end{defn}

We will focus on irreducible dual pairs, because these are the building blocks of all the others.  Such pairs are of the following kinds: 
\begin{itemize}
\item[$\bullet$] pairs of type I:
\begin{itemize}
\item[$\circ$] $(\Sp_{2m}(k),\rO_{n'}(k))$ with $mn'=N$;
\item[$\circ$] $(\rU_n(k),\rU_{n'}(k))$ with $nn'=2N$;
\end{itemize}
\item[$\bullet$]
pairs of type II: 
\begin{itemize}
\item[$\circ$] $(\GL_n(k),\GL_{n'}(k))$ with $nn'=2N$.
\end{itemize}
\end{itemize}

The dual pairs formed by a symplectic group and an orthogonal group are more precisely described as follows.
\begin{itemize}
\item  $n$ and $n'$ be two fixed non-negative integers
\item $V=V_n$ a $k$-vector space of (even) dimension $n=2m$, endowed with a non-degenerate alternated
$k$-bilinear form $\langle \,,\,\rangle$.
The corresponding group of isometries  is the symplectic group \textit{$G:=\Sp(V_n)$}.
\item $V'=V'_{n'}$  a  $k$-vector space of dimension $n'$ over $k$, endowed with a non-degenerate symmetric
$k$-bilinear form $\langle \,,\,\rangle'$.
The orthogonal group \textit{$G':=\rO(V_{n'})$} is the corresponding group of
isometries.
\item  $(G,G')$ is a dual pair in $\Sp(W)$ where $W:=V_n\otimes_k V'_{n'}$ with
\[\langle\langle v_1\otimes v_1',v_2\otimes v_2'\rangle\rangle:=\langle v_1,v_2\rangle\cdot  \langle v_2',v_1'\rangle',\]
for $v_1,v_2\in V$ and $v_1',v_2'\in V'$.
\end{itemize}

\begin{defn} \label{defn:stable-range} {\rm
A reductive dual pair $(G,G')$ is said to be in the \textit{stable range} (with $G'$ smaller) if the defining vector space for $G$ has a totally isotropic subspace of dimension 
greater or equal than the dimension the defining vector space for $G'$.
}
\end{defn}

\begin{examples} {\rm
The pair $(\Sp_{2m}(k),\rO^\epsilon_{2m'}(k))$ is in the stable range if and only if  $m\ge 2m'$. 
Consider now a pair  $(G_n,G_{n'}')=(\rO^\epsilon_{2m'}(k),\Sp_{2m}(k))$, with $\rO_{2m'}(k)=\rO(V_{2m'})$. The maximum dimension of an isotropic subspace of $V'_{2m'}$ with respect to $\langle\,,\,\rangle'$ is 
either $m'$ or $m'-1$, giving rise to the two orthogonal groups $\rO^+_{2m'}(k)$ and $\rO_{2m'}^-(k)$, respectively. Thus, the pair $(\rO^+_{2m'}(k),\Sp_{2m}(k))$ 
is in the stable range if and only if  $m'\ge 2m$, and $(\rO^-_{2m'}(k),\Sp_{2m}(k))$ is in the stable range if and only if  $m'-1\ge 2m$. 
 
Note that, if $\rO(V_{n'})=\rO_{2m'+1}(k)$, then the maximum dimension of an isotropic subspace of $V'_{2m'+1}$ is $m'$, and the pair $(\rO_{2m'+1}(k),\Sp_{2m}(k))$ is in the stable range 
if and only if  $m'\ge 2m'$.}
\end{examples}

We will introduce a few more notation in order to be able consider all possible irreducible reductive dual pairs. Let $\xi$  denote the unique quadratic character of $k^\times$, that is,
\begin{equation}
\xi\colon k^\times \to \{\pm 1\}, \quad a\mapsto a^{\frac{q-1}{2}}.
\end{equation}
Let $\tau$ be an element in $\Gal(k/\FF_p)$ such that $\tau^2$ is the identity, and let $k_0$ denote the fixed field of $\tau$.

For $\varepsilon\in\{\pm 1\}$, a \textit{$\varepsilon$-Hermitian space} $V$ (over $k$) is a $k$-vector space equipped with a
non-degenerate form $\langle\;,\;\rangle_V\colon V\times V\to k$ such that
\begin{equation}
\langle av_1+v_2,v_3\rangle_V =a\langle v_1,v_3\rangle_V +\varepsilon \langle v_3,v_2\rangle_V^\tau , \quad\text{for all $a\in k$, and $v_1,v_2,v_3 \in V$.}
\end{equation}
Let $m=\dim V$ and let
\begin{equation}
G(V):=\left\{g\in\End_k(V)\;:\;\langle g(v_1),g(v_2)\rangle_V =\langle v_1,v_2\rangle_V ,\;\text{ for all $v_1,v_2 \in V$}\right\}
\end{equation}
be the isometry group of  the form $\langle\cdot,\cdot\rangle_V$. If $k=k_0$ and $\varepsilon = 1$, we choose a basis $\{e_1,\ldots ,e_m\}$ of $V$ and define the \textit{discriminant} of $V$ by
\begin{equation}
\disc(V):=(-1)^{\frac{m(m-1)}{2}}\det\left(\langle e_i,e_j\rangle_V\right)\in k^\times/k^{\times 2}.
\end{equation}
Note that $\disc(V)$ is independent of the choice of the basis $\{e_1,\ldots ,e_m\}$.

\begin{remark}
{\rm The group $\rO_m^\epsilon$ is the isometry group of $V$ with $\disc(V)=\epsilon\, 1$. For even $m$, the  orthogonal group $\rO_m^\epsilon$ is split if and only if $\epsilon=+$. For odd $m$,
there is a unique orthogonal group $\rO_m$ as an abstract groups; however, it acts on two symmetric spaces with different discriminants.}
\end{remark}

\smallskip

According to different choices of $k$ and $\varepsilon$, we have the following cases:
\begin{enumerate}
\item
$k\ne k_0$:  $(G,G')=(G(V ),G(V'))$ is a unitary dual pair;
\item
$k=k_0$, $\varepsilon=1$ and $\dim V$ is odd: $(G,G')=((G(V ),G(V'))$  is an odd orthogonal-symplectic dual pair;
\item
$k=k_0$,  $\varepsilon=-1$ and $\dim V'$ is even: $(G,G')=((G(V ),G(V'))$  is a symplectic-even orthogonal dual pair;
\item
$k=k_0$, $\varepsilon= -1$ and $\dim V'$ is odd: $(G,G')=((G(V ),G(V'))$  is a symplectic-odd orthogonal dual pair. 
\item
$k=k_0$,  $\varepsilon=1$ and $\dim V$ is even: $(G,G')=((G(V ),G(V'))$ is an even orthogonal- symplectic dual pair.
\end{enumerate}

\subsection{Definition of the theta correspondence}\

Let $\psi$ be a nontrivial additive character of $k$.
G\'erardin introduced in  \cite[Theorem~2.4]{Ger} a slightly different  version $\omega_\psi^\flat$  of $\omega_\psi$.

When $(G,G')$ is a pair formed by a symplectic and an orthogonal group, the representations $\omega_\psi$ and $\omega_\psi^\flat$ have the same restriction to $G\cdot G'$.
In the other cases, the restrictions of $\omega_\psi$ and $\omega_\psi^\flat$ differ by a representation of $G$ with values in $\{\pm 1\}$. 

From now on we suppose that the character $\psi$ is fixed. Let $\omega^\flat_{G,G'}$ denote the restriction of $\omega^\flat_\psi$ to $G\cdot G'$.

\smallskip

Let $(G,G'):=(G(V),G(V')$. In \cite{MQZ}, a modified version $\omega_{G,G'}$ of $\omega_{G,G'}^\flat$ is used. It  is defined as follows:
\begin{equation} \label{eqn:omegaGGdef}
\omega_{G,G'}:=
\begin{cases}
\left((\xi\circ \det_{G})^{\frac{\dim V'}{2}}\otimes 1_{\G'}\right)\otimes\omega^\flat_{G,G'}&\text{in cases (2) and (5)}\cr
\left(1_{G}\otimes(\xi\circ \det_{G'})^{\frac{\dim V}{2}}\right)\otimes\omega^\flat_{G,G'}&\text{in cases (3) and (4)}.
\end{cases}
\end{equation}
The representation $\omega_{G,G'}$ has the form
\begin{equation} \label{eqn:omegaGG}
\omega_{G,G'}=\sum_{\pi\in\Irr(G)\atop\pi'\in\Irr(G')}\mult_{\pi,\pi'}\,\,\pi\otimes\pi',\;\;\text{where $\mult_{\pi,\pi'}\in \ZZ_{\ge 0}$}.
\end{equation}
We set
\begin{equation} \label{eqn:thetaGG}
\Theta_{G,G'}:=\left\{(\pi,\pi')\in\Irr(G)\times\Irr(G')\,:\,\mult_{\pi,\pi'}\ne 0\right\}.
\end{equation}
Hence $\Theta_{G,G'}$ defines a \textit{correspondence} between $\Irr(G)$ and $\Irr(G')$.  This correspondence is \textit{symmetric}, that is, $(\pi,\pi')\in\Theta_{G,G'}$ if and only if $(\pi',\pi)\in\Theta_{G',G}$.

\begin{defn} \label{defn:Theta} {\rm Let $\pi\in\Irr(G)$. We say that the representation \textit{$\pi$ occurs in the  theta correspondence for the dual pair $(G,G')$} if  there exists $\pi'\in\Irr(G')$ such that $(\pi,\pi')\in\Theta_{G,G'}$.}
\end{defn}

For $\pi\in\Irr(G)$, we  define its \textit{theta lift}
\begin{equation}\theta_{G'}(\pi):=\bigoplus_{\text{$\pi'\in\Irr(G')$ such that $(\pi,\pi')\in \Theta_{G,G'}$}}\pi',
\end{equation}
and we extend $\theta_{G'}$ by linearity into a map 
\begin{equation} \label{eqn:thetaG}
\theta_{G'}\colon \ZZ\,\Irr(G)\to \ZZ\,\Irr(G').
\end{equation}

In general, if nonzero, $\theta_{G'}(\pi)$ is reducible, and the main goal of these notes is to obtain an explicit description of the representation $\theta_{G'}(\pi)$.

On the other side, since the correspondence $\Theta_{G,G'}$ is in general not one-to-one, a natural question is: can we find an interesting sub-relation of $\Theta_{G,G'}$ which is one-to-one.

\subsection{Witt towers of theta correspondences}\

The non-vanishing question of theta lifts for (Type I dual pairs) can be approached by considering theta lifts in the Witt tower of the the target space.

Let $\bbH$ denote the unique (up to isomorphism) $2$-dimensional $\varepsilon'$-Hermitian vector space containing a non-zero isotropic vector. We call $\bbH$  the hyperbolic $\varepsilon'$-Hermitian plane.

\begin{defn} {\rm 
 A set $\cT'$ of $\varepsilon'$-Hermitian vector spaces is called a \textit{Witt tower} if there is an anisotropic  $\varepsilon'$-Hermitian vector  space $V'_\an$  such that $\cT'$ consists of all $\varepsilon'$-Hermitian vectors spaces $V'(r')$
 satisfying
\[V'(r')\simeq V'_\an\oplus \underbrace{\bbH\oplus\cdots\oplus\bbH}_{r'}.\] 
The integer $r'$ above is called the \textit{split rank} of  $V'(r')$ and the parity of $\dim V'_\an$ is called the \textit{parity of $\cT'$}.
}
\end{defn}
If $k=k_0$ and $\varepsilon'=1$, we have $\disc(V'(r'))=\disc( V'_\an)$ for every integer $r'\ge 0$.
Therefore, we define 
\begin{equation}
\disc(\cT'):= \disc( V'_\an),
\end{equation} and call it the discriminant of $\cT'$ in this case.
\begin{remark} \

{\rm
\begin{enumerate}
\item For even orthogonal groups there are two Witt towers depending if $G'_{n'}=\rO(V_{n'})$ is split or not. We denote them $\rO^+$ and $\rO^-$, respectively, and write $\rO_{2n}^\epsilon(k)$ for the corresponding orthogonal group
\item
For odd orthogonal groups there are two Witt towers as well $\left\{\rO_{2n+1}^\epsilon(k)\right\}_{n\ge 0}$, with $\epsilon=\pm$.
\item
For unitary groups there are two Witt towers $\rU^+:= \left\{\rU_{2m}(k)\right\}_{n\ge 0}$ 
and $\rU^-: = \left\{\rU_{2m+1}(k)\right\}_{n\ge 0}$. We will set $\rU^+_n(k):=\rU_{2m}(k)$ if $n=2m$ and  $\rU^-_n(k):=\rU_{2m+1}(k)$ if $n=2m+1$.
\end{enumerate}}
\end{remark}

For a reductive dual pair $(G,G')=(G(V),G(V')$, let $\cT'$ be the Witt tower containing $V'$. We define the generalized Witt tower $\cT_G'$ to be the collection
\begin{equation}
\cT'_G:=\left\{(V'(r'),\omega_{G(V),G(V'(r'))})\;:\,r'\in \ZZ_{\ge 0}\right\}.
\end{equation}

The \textit{companion generalized Witt tower of  $\cT'_G$} is defined as follows:
\begin{itemize}
\item[$\bullet$]
In case (1), there exists a unique Witt tower 
$\widetilde\cT':=\left\{\widetilde V'(r')\;:\,r'\in \ZZ_{\ge 0}\right\}$ such that  $\cT'$ and $\widetilde\cT'$ have a different parity. We set
\[\widetilde\cT'_G:=\left\{(\widetilde V',\omega_{G,\rU(\widetilde V')})\;:\,\widetilde V'\in\widetilde\cT'\right\}.\]
\item[$\bullet$]
In cases (3) and (4), there exists a unique Witt tower $\widetilde\cT'':=\left\{\widetilde V'(r')\;:\,r'\in \ZZ_{\ge 0}\right\}$ such that $\cT'$ and $\widetilde\cT'$ have the same parity and different discriminants. We set
\[\widetilde\cT'_G:=\left\{(\widetilde V',\omega_{G,\rU(\widetilde V')})\;:\,\widetilde V'\in\widetilde\cT'\right\}.\]
\item[$\bullet$]
In cases (2) and (5), there is only one Witt tower of symplectic spaces.  We write $\widetilde\cT':=\cT'$ and set
\[\widetilde\cT'_G:=\left\{(\widetilde V',\omega_{G,G(V')}\otimes (\deter_{G}\otimes 1_{G(V')}))\;:\, V'\in\cT'\right\}.\]
\end{itemize}

\begin{defn} \label{defn:first-occurr} {\rm 
Let $\pi\in\Irr(G)$.
\begin{itemize}
\item 
The smallest non-negative integer $n'_{\cT'}(\pi)$ such that $\theta_{n'_{\cT'}(\pi)}(\pi)\ne 0$ is called the \textit{first occurrence index of $\pi$ with respect to the Witt tower $\cT'$}.
\item The representation 
\[\theta^0_{\cT'}(\pi):=\theta_{n'_{\cT'}(\pi)}(\pi)\] is called the \textit{first occurrence of $\pi$ with respect to  $\cT'$}.
\end{itemize} 
If the Witt tower $\cT'$, is fixed, we will  simply write $\theta^0(\pi)$ for $\theta^0_{\cT'}(\pi)$.}
\end{defn}

Similarly, for $G'=G_{n'}'$ with $n'$ fixed, one can then consider the tower of the theta correspondences associated to the tower of reductive dual pairs $(G_{m}, G')_{m\ge 0}$. 
For every $\pi'\in\Irr(G')$, we have the representation $\theta_{m'}(\pi)$. 

The smallest non-negative integer $n_{\cT}'$ such that $\theta_{n_{\cT}(\pi')}(\pi')\ne 0$ is called the \textit{first occurrence index of $\pi'$ with respect to the Witt tower $\cT$}, and the representation  
\[ \theta^0_{\cT}(\pi'):=\theta_{n_{\cT}(\pi')}(\pi')\]  the \textit{first occurrence of $\pi'$ with respect to $\cT$}. 

\begin{prop}
For every $\pi\in\Irr(G)$ and $\pi'\in\Irr(G')$, such integers $n'_{\cT'}(\pi)$ and $n_{\cT}(\pi')$ exist, and we have $n_{\cT'}'(\pi)\le n'$ and $n_{\cT}(\pi')\le n$. 
Moreover, we have $\theta_{m'}(\pi)\ne 0$ for any $m'\ge n'_{\cT'}(\pi)$ 
and $\theta_{m}(\pi')\ne 0$ for any $m\ge n_{\cT}(\pi')$.
\end{prop}

\section{Theta correspondence and Harish-Chandra series} \label{sec:thetaHC}

\subsection{Harish-Chandra series of finite groups of Lie type}
Let $k$ be  a finite field and let $\overline{k}$ be a fixed algebraic closure of $k$, and let $\bG$ be a connected reductive algebraic group defined over $k$, and let $G=\bG(k)$ the group of the $k$-rational points of $\bG$. We denote by $(\;,\;)_{G}$  the usual scalar product on the space of class functions on $G$: 
\begin{equation} \label{eqn:scalar product}
(f_1,f_2)_{G} := |G|^{-1}\sum_{g\in G}f_1(g)\,\overline{f_2(g)} .
\end{equation}
Let $\Irr(G)$ denote the set of equivalence classes of irreducible representations of $G$. For $\pi\in\Irr(G)$, we will also write $\pi$ for its character.

Let $P=L\ltimes U$ be a parabolic subgroup of $G$ with Levi factor $L$ and unipotent radical $U$.

\begin{defn} \
{\rm
\begin{itemize}
\item The \textit{parabolic induction functor} is the composition of the inflation functor $\infl_L^P$ with the induction $\Ind_P^G$:
\[\ii_{L,P}^G\colon\fR(L)\xrightarrow{\infl_L^P}\fR(P)\xrightarrow{\Ind_P^G}\fR(G).\]
\item The \textit{parabolic restriction functor} is the composition of the restriction functor $\Res_{P}^G$ with the functor of $U$-(co)invariants:
\[ \rr_{L,P}^G\colon\fR(G)\xrightarrow{\Res_L^P}\fR(P)\xrightarrow{\text{$U$-coinvariants}}\fR(L).\]
\end{itemize}
}
\end{defn}

The functors $\ii_{L,P}^G$ and $\rr_{L,P}^G$ are both exact (see \cite[Corollary 3.1.6]{GMbook}), and $r_{L,P}^G$ is the left (and right) adjoint of $\ii_{L,P}^G$.

\begin{defn} {\rm
A representation $\pi$ of $G$ is \textit{cuspidal} if $\rr_{L,P}^G(\pi)=0$ for any proper parabolic subgroup $P$ of $G$.

Actually, $\ii_{L,P}^G$ and $\rr_{L,P}^G$ do not depend on the choice of of the parabolic $P$ with Levi subgroup $L$ (see \cite[Theorem 5.3.1]{DMbook} or \cite[Theorem 3.1.13]{GMbook}). Hence, from now one, we will simply denote them by $\ii_{L}^G$ and $\rr_{L}^G$.

Let $\Irr_\cc(G)$ denote the set of isomorphism classes of cuspidal irreducible representation of $G$.}
\end{defn}

\begin{thm} 
Every $\pi\in\Irr(G)$  occurs as an irreducible component of a parabolically induced representation $\ii_{L}^G(\sigma)$, where $L$ the Levi factor $L$ of a parabolic subgroup of $G$ and $\sigma$ is a cuspidal irreducible  representation of $L$.

\smallskip
\noindent
The  $G$-conjugacy class $(L,\sigma)_G$ of $(L,\sigma)$ is uniquely determined and is  called the \textit{cuspidal support} of $\pi$. It is denoted by $\Sc(\pi)$.
\end{thm}
\begin{proof} See  \cite[\S5.3]{DMbook} or \cite[Corollary 3.1.17]{GMbook}.
\end{proof}

\begin{defn}{\rm
Let $\sigma\in\Irr_\cc(L)$. We set
\begin{equation} \label{eqn:HCs}
\Irr^\sigma(G):=\left\{\pi\in\Irr(G)\,:\,\sigma\in\Sc(\pi)\right\}.
\end{equation}
The set $\Irr^\sigma(G)$ is the \textit{Harish-Chandra series} of $\sigma$.}
\end{defn}

Let $\sigma\in\Irr_\cc(L)$. The group
\begin{equation} \label{eqn:Wsigma} 
W_{\sigma}:=\left\{n\in\rN_G(L)/L\,;\,{}^w\sigma\simeq\sigma\right\}
\end{equation}
decomposes as 
\begin{equation} \label{eqn:Dec} 
W_{\sigma}=W_{\sigma}^\circ\rtimes R_{\sigma},
\end{equation}
where $W_{\sigma}^\circ$ is a finite Weyl group, with root system denoted by $\Phi_{\sigma}$ and 
\begin{equation} \label{eqn:Rsigma} 
R_{\sigma}:=\left\{w\in W_{\sigma}\,:\, w \Phi_{\sigma}^+\subset \Phi_{\sigma}^+\right\}.
\end{equation}

\begin{remark}{\rm If $G$ has connected centre, then $W_{\sigma}$ is a Coxeter group, that is,  $W_{\sigma}=W_{\sigma}^\circ$ (see \cite[Chapter~8]{Lubook}).}
\end{remark}

\begin{thm} \label{thm:HC:W}
Let $\sigma\in\Irr_{\cc}(L)$. 
The set $\Irr_\sigma(G)$ is in bijection with the isomorphism classes of irreducible representations of $W_{\sigma}$.
\end{thm}
\begin{proof}
It follows from the combination of \cite[Theorem 3.1.18]{GMbook} and \cite[Theorem 3.2.5]{GMbook}.
\end{proof}

\subsection{Application to the theta correspondence} \label{subsec:HCT}\

Let $G_n\in\{\Sp_{2n}(k), \rO^\epsilon_{2n}(k), \rO_{2n+1}(k), \rU_m\}$. 

\begin{thm}
If $\pi\in\Irr_\cc(G)$ then its first occurence (with respect to each Witt tower) is also cuspidal:
\[\pi\in\Irr_\cc(G)\;
\Rightarrow\;\theta^0(\pi)\in\Irr_\cc(G'_{n'_{\cT}(\pi)}).\]
\end{thm}
\begin{proof}
See \cite[Th\'eor\`eme.7]{AMR}.
\end{proof}

A Levi subgroup $L=L_n$ of a parabolic subgroup $P=P_n$ of $G=G_n$ is isomorphic to 
\begin{equation} \label{eqn:Levi}
G_{n_0}\times \GL_{n_1}(k)\times\cdots\times\GL_{n_r}(k)\quad\text{where $n_0+n_1+\cdots+n_r=n$.}
\end{equation}
Whenever $P$ is maximal, that is, whenever $r=1$, for $\pi\in\Irr(G)$, we set
\begin{equation}
c(\pi):=\max\left\{m\,:\,\Hom_{\GL_{m_1}(k)}(\rr^G_{L}(\pi),1)\ne 0\right\}.
\end{equation}

\begin{prop} The following \textit{conservation relation}  on the first occurrence indices holds
\[n'_{\cT_+'}(\pi)+n'_{\cT_-'}(\pi)+c(\pi)=4n+2.\]
\end{prop}
\begin{proof}
See \cite[Proposition 3.7]{MQZ}.
\end{proof}

\begin{prop} 
If  $L$ is a Levi subgroup of $G$ such that $n_1=n_2=\cdots=n_r=1$, that is, 
\begin{equation} \label{eqn:LeviMQZ}
L\simeq G_{n_0}\times \GL_{1}(k)^r,\quad\text{with $n_0+r=n$},
\end{equation}
then, for every  $\sigma\in \Irr_\cc(L)$, the group $W_\sigma$ is a finite Weyl group of type $\rB_r$.
\end{prop}
\begin{proof}
It is a consequence of  \cite[Theorem~1.2]{MQZ}.
\end{proof}

Let $\sigma\in\Irr_\cc(L)$ and $\pi\in\Irr^\sigma(G)$. We have
\begin{equation}
\sigma\simeq\pi_{n_0}\otimes\tau_1\otimes\cdots\otimes\tau_r,
\end{equation}
where $\pi_{n'_0}\in\Irr_\cc(G_{n_0})$ and $\tau_i\in\Irr_\cc(\GL_{m_i}(k))$ for $1\le i\le r$.

A Levi subgroup $L'=L_{n'}'$ of a parabolic subgroup $P'=P_{n'}'$ of $G'=G_{n'}'$ is isomorphic to 
\begin{equation}
G'_{n'_0}\times \GL_{n'_1}(k)\times\cdots\times\GL_{n'_{r'}}(k)\quad\text{where $n'_0+n'_1+\cdots+n'_{r'}=n'$.}
\end{equation}
Let $\pi'\in\Irr_{\sigma'}(G')$ with
\begin{equation}
\sigma'\simeq\pi_{n'_0}'\otimes\tau_1'\otimes\cdots\otimes\tau_{r'}'\in\Irr_\cc(L'_{n'})
\end{equation}
such that $\pi_{n'_0}'\in\Irr_\cc(G_{n'_0}')$ and $\tau_j'\in\Irr_\cc(\GL_{n'_j}(k))$ for $1\le j\le r'$. 

\begin{thm} \label{thm:HC-HC}
Let $(G,G')=(G_n,G'_{n'})$ be a reductive dual pair, let $\sigma\in\Irr_\cc(L)$, and let $\pi\in\Irr^\sigma(G)$. 

Then every $\pi'\in\theta_{G'}(\pi)$ belongs to $\Irr^{\sigma'}(G')$, where
\begin{enumerate} 
\item if $\overline{n}\le\overline{n}'$
\[\sigma'=\theta^0(\pi_{n_0})\otimes\tau_1\otimes\cdots\otimes\tau_r\otimes 1\otimes\cdots\otimes 1;\]
\item otherwise there exists a sequence $i_1$, $\ldots$, $i_t$ such that
\[\sigma'=\theta^0(\pi_{n_0})\otimes\tau_1\otimes\cdots\otimes\widehat{\tau_{i_1}}\otimes\cdots\otimes\widehat{\tau_{i_t}}\otimes\cdots\otimes\tau_r,\]
\end{enumerate}
with $\theta^0(\pi_{n_0})$ the first  occurrence of $\pi_{n_0}$ with respect to the Witt tower associated to $G'$.
\end{thm}
\begin{proof}
See \cite[Th\'eor\`eme~3.7]{AMR}.
\end{proof}

As an immediate consequence of Theorem~\ref{thm:HC-HC}, we have:
\begin{cor} \label{cor:HC}
If $L$ is of the form \eqref{eqn:LeviMQZ}, and $\sigma=\pi_{n_0}\otimes \chi^r$, then we have
\[L'\simeq G'_{n'_0}\times \GL_{1}(k)^{r'}\quad\text{and $\sigma'=\theta^0(\pi_{n_0})\otimes\chi^{r'}$},\]
where $n'_0=n_{\cT'}(\sigma_0)$ and $n'_0+r'=n'$.
\end{cor}

\section{Theta correspondence and Lusztig series} \label{sec:Lusztig}
\subsection{Deligne-Lusztig virtual characters}\

Given $\bX$ a $\overline k$-variety and $\ell$ a prime number not dividing $q$, Grothendieck constructed \textit{$\ell$-adic cohomology groups with compact support} $H_c^i(\bX,\Qlbar)$ which are finite dimensional $\Qlbar$-vector spaces.
We have $H_c^i(\bX,\Qlbar)=$ if $i\notin\{0,1,\cdots,2\dim\bX\}$.

Let $\bU$ be the unipotent radical of a parabolic subgroup $\bP$ of $\bG$. We consider the variety 
\begin{equation}
\bX_\bU:=\left\{g\bU \in G/\bU \;:\;g\bU \cap F (g\bU)\ne\emptyset\right\}.
\end{equation} 
It has an obvious left $G$-action, and a right $L$-action since $\bL$ normalizes $\bU$.

\begin{defn}\
 {\rm 
\begin{enumerate}
\item The \textit{Lustig induction} $R_\bL^\bG(M)$ of a $\Qlbar[L]$-module $M$ is defined by
\[R_\bL^\bG(M):=H_c^i(\bX_\bU,\Qlbar)\otimes_{\Qlbar[L]} M.\]
\item The \textit{Lustig restriction} ${}^\ast R_\bL^\bG(N)$ of a $\Qlbar[G]$-module $N$ is defined by
\[{}^\ast R_\bL^\bG(N):=N \otimes_{\Qlbar[G]} H_c^i(\bX_\bU,\Qlbar).\]
\end{enumerate}}
\end{defn}

When $\bL$ is a torus $\bT$ and $\theta\in\Irr(T)$ the virtual characters $R_\bL^\bG(\theta)$  where introduced in \cite{DL} and are called \textit{Deligne-Lusztig characters}. 

\begin{defn} \label{defn:uniform}
 {\rm A  complex representation of $G$ is  \textit{uniform} if its character is a linear combination of the Deligne-Lusztig characters.}
\end{defn}

Let $W$ be the Weyl group of a quasi-split torus of $\bG$. We consider the coset $WF$ of $W$ in $W\rtimes\langle F\rangle$, where $\langle F\rangle$ is the group generated by the automorphism of finite order induced by $F$ on $W$.
Let $f\colon WF\to\CC$ be a function which is invariant by $W$-conjugacy. We define
\begin{equation} \label{eqn:Rf}
R_f:= :=\frac{1}{|W|}\sum_{w\in W} f(wF) R^\bG_{\bT_w}(1).
\end{equation}

For unitary groups and general linear groups, every representation is uniform (see \cite[Theorem~11.7.3]{DMbook}). 

From now on for $\pi\in\Irr(G)$, we will also write $\pi$ for its character.

\begin{thm} {\rm \cite[Corollary~7.7]{DL}}
Let $\pi\in\Irr(G)$. Then, there exists a $k$-rational maximal torus $\bT$ of $\bG$ and a character $\theta$ of $T$ such that $(\pi, R_\bT^\bG(\theta))_{G}\ne 0$.
\end{thm}

In \cite{DL} the following two regularity conditions for a character $\theta$ of $T$ are defined:
\begin{enumerate}
\item $\theta$ is said to be \textit{in general position} if its stabilizer in $(\Nor_\bG(\bT)/\bT)(k)$ is trivial. 
\item $\theta$ is said to be \textit{non-singular} if it is not orthogonal to any coroot.
\end{enumerate} 
If the centre of $\bG$ is connected, then $\theta$ is non-singular if and only if it is in general position (see Proposition~5.16 in \cite{DL}).

We will need to extend the above definitions to the case where the group $\bG$ is possibly disconnected. We denote by $\bG^\circ$ the identity component of $\bG$. For $\bT$ a $k$-rational maximal torus of $\bG^\circ$ and $\theta$ a character of $T=\bT(k)$, we set
\begin{equation}
R_{\bT}^{\bG}(\theta):=\Ind_{\bG^{\circ}(k)}^{\bG(k)}(R_{\bT}^{\bG^{\circ}}(\theta)).
\end{equation}

\begin{defn} \label{defn:unipotent}
{\rm The representation $\pi$ of $G$ is \textit{unipotent} if $(\pi, R_\bT^\bG(1))_{G}\ne 0$, for some $k$-rational maximal torus $\bT$ of $\bG$.
Let $\Irr^\unip(G)$ denote the set of isomorphism classes of unipotent irreducible representations of $G$. }
\end{defn}

We write
\begin{equation} \label{eqn:unip-cusp}
\Irr_{\cusp,\unip}(G):=\Irr_\cusp(G)\cap\Irr^\unip(G).
\end{equation}

\subsection{Lusztig series}\

We denote by $\check \bG^{\circ}$ the reductive connected group with root datum dual to that of $\bG^\circ$.
The $G^\circ$-conjugacy classes of pairs $(\bT,\theta)$ as above are in one-to-one correspondence with the $\check G^{\circ}$-conjugacy classes of pairs $(\check \bT,s)$, where  $\check \bT$ is a $k$-rational maximal torus of $\check \bG^{\circ}$ and $s$ is a semisimple element of $\check G^{\circ}$ belonging to  $\check \bT$. We then write $(\bT,\theta)\leftrightarrow (\check \bT,s)$.

\begin{defn}{ \label{defn-Deligne-Lusztig-series} The \textit{Lusztig series} are the sets
\[ \cE(G,s):=\left\{\pi\in\Irr(G)\colon
  \text{$\pi$ occurs in $R_\bT^\bG(\theta)$, where $(\bT,\theta)_{G^\circ}$ corresponds to $(\check \bT,s)_{\check G^{\circ}}$}\right\},
\]
where $(s)$ runs over  the $\check G^{\circ}$-conjugacy classes of semisimple elements of $\check G^{\circ}$.}
\end{defn}
Lusztig series provide a partition of the set $\Irr(G)$:
\begin{equation} \label{eqn:Lusztig-series}
\Irr(G)=\bigsqcup_{(s)}\cE(G,s),
\end{equation}
where $(s)$ runs over  the $\check G^{\circ}$-conjugacy classes of semisimple elements of $\check G^{\circ}$.  The subset $\cE(G,1)$ is the set of unipotent representations in $\Irr(G)$. 

Let $\bG_s$ denote the dual group of the centralizer of $s$ in $\check \bG$. We set $G_s:=\bG_s(k)$.
By  \cite[Theorem~4.23]{Lubook} when  $\bG$ is connected group with connected center, by \cite[\S12]{Lus-disconnected-center} when $\bG$ is an arbitrary connected group, and by \cite[Proposition~1.7]{AMR} when $\bG=\rO^\pm_{2n}$, there is a bijection
\begin{equation}\label{Lus-unipotent-decomposition}
   \cL_s^G\colon \mathcal{E}(G,s)\xrightarrow{1-1} \mathcal{E}(G_s,1)=\Irr^\unip(G_s), \quad
    \pi\mapsto\pi^\unip .
\end{equation}
The bijection \eqref{Lus-unipotent-decomposition} satisfies the following properties:
\begin{enumerate}
    \item[{\rm (1)}] It sends a Deligne-Lusztig character $R_\bT^\bG(\theta)$ in $G$ (up to a sign) to a Deligne-Lusztig character $R_{\bT}^{\bG_s}(1)$, where $1$ denotes the trivial character of $T$ (see \S12 of \cite{Lus-disconnected-center}):
 \begin{equation} \label{eqn:ee}
 \cL_s^G(\varepsilon_G R_\bT^\bG(\theta))=\varepsilon_{G_s}R_{\bT}^{\bG_s}(1),
 \end{equation}
for  $(\bT,\theta)\leftrightarrow (\check \bT,s)$, where $\varepsilon_G:=(-1)^{\rk(\bG)}$, with $\rk(\bG)$ denoting the $k$-rank of $\bG$ (see \cite[Definition 7.1.5]{DMbook}).
  \item[{\rm (2)}] It preserves cuspidality in the following sense: 
if $\pi\in\cE(G,s)$ is cuspidal, then
\begin{itemize}
    \item[{\rm (a)}]
the largest $k$-split torus in the center of $\bG_s$ coincides with the largest $k$-split torus in the center of $\bG$ (see \cite[(8.4.5)]{Lubook}),
    \item[{\rm (b)}] the unipotent representation $\pi^\unip$ is cuspidal.
 \end{itemize}
    \item[{\rm (3)}]
The dimension of $\pi$ is given by
\begin{equation} \label{eqn:dim tau}
\dim(\pi)=\frac{|G|_{p'}}{|\bG_s(k)|_{p'}}\,\dim(\pi^\unip), 
\end{equation}
where $|G|_{p'}$ is the largest prime-to-$p$ factor of the order of $G$ (see \cite[Remark~13.24]{DMbook}).
\end{enumerate}

By collecting the bijections $\cL_s^G$ for all the semisimple element $s$, we obtain a bjijection 
\begin{equation} \label{eqn:Jordan}
\cL^G\colon\Irr(G)\longrightarrow \bigcup_{s}\Irr^\unip(G_s).
\end{equation}

\begin{remark}\

 {\rm 
\begin{enumerate}
\item When the center of $\bG$ is connected, the bijection $\cL^G$ is uniquely determined by the properties above (see \cite{DM90}). It applies in particular when $\bG=\SO_{2n+1}$.
\item The restriction of $\cL^G$ to uniform representations of $G$ is uniquely determined by \eqref{eqn:ee}. Hence, $\cL^G$ is uniquely determined  by the properties above, whenever $\bG$ is a general linear group or a unitary group, since every irreducible representation of $\bG$ is uniform.
\end{enumerate}
}
\end{remark}

Moreover, for the other classical groups $G$ the following holds.

\begin{thm}
Let $\bG\in\{\Sp_{2n},\rO_{2n}^\epsilon,\rO_{2n+1}^\epsilon\}$.
Suppose that $q$ is large enough so that the main result in \cite{Sri} holds. 
Then, for each group $G$, there is a unique choice of $\cL^G$. This choice is compatible with parabolic induction and the theta correspondence (defined with respect to $\omega_{G,G'})$.
\end{thm}
\begin{proof}
See \cite[Theorem~6.1]{WangJ}.
\end{proof}

\subsection{Decompositions of the endoscopic groups} \label{subsec:centralizers}\

Let $\bG$ be  a classical group of rank $m$ over $\overline{k}$, and $\bT_m\simeq\overline{k}^\times\times\cdots\times\overline{k}^\times$ be a $k$-rational maximal torus of $\bG$. For $s = (\lambda_1,\ldots,\lambda_m)\in \check \bT_m$,  we denote by let $\nu_\lambda(s)$ denote the number of the $\lambda_i$'s which are equal to $\lambda$, and  by $\langle\lambda\rangle$ the set of all roots in $\overline{k}$ of the irreducible polynomial of $\lambda$ over $k$.  

The endoscopic group $\bG_s$  decomposes as a product
 \begin{equation} \label{eqn:centr}
 \bG_s=\prod_{\langle\lambda\rangle\subset\{\lambda_1,\ldots,\lambda_m\}}\bG_{s,\lambda},
 \end{equation}
where $\bG_{s,\lambda}$ is a reductive quasi-simple group of rank equal to $\nu_\lambda(s)\cdot |\langle\lambda\rangle|$, see for instance \cite[\S1.B]{AMR}. 

\smallskip

Let $\bG\in\{\SO_{2n+1}, \Sp_{2n}, \rO^\pm_{2n}, \rU_m\}$. Then the following holds:
\begin{itemize}
\item[(i)] if $\bG=\SO_{2n+1}$, then $\check \bG=\Sp_{2n}$, $\bG_{s,1}\simeq\SO_{2\nu_1(s)+1}$, and $\bG_{s,-1}\simeq\SO_{2\nu_{-1}(s)+1}$;
\item[(ii)] if $\bG=\Sp_{2n}$, then $\check \bG=\SO_{2n+1}$, $\bG_{s,1}\simeq \Sp_{2\nu_1(s)}$ , and $\bG_{s,-1}\simeq\rO^\pm_{2\nu_{-1}(s)}$;
\item[(iii)] if $\bG=\rO^\pm_{2n}$, then $\check \bG=\rO^\pm_{2n}$, $\bG_{s,1}\simeq \rO^\pm_{2\nu_1(s)}$, and $\bG_{s,-1}\simeq \rO^\pm_{2\nu_{-1}(s)}$;
\item[(iv)] if $\bG=\rU_m$ or if $\lambda\ne\pm 1$, then $\bG_{s,\lambda}$ is either a general linear group or a unitary group.
\end{itemize}

\subsubsection{Unitary groups}\

When $\bG=\rU_m$,  the endoscopic group $ \bG_s$ decomposes as 
\begin{equation}
 \bG_s=\bG_{s,1}\times \bG_{s,\ne 1},
 \end{equation}
 where
\begin{equation}
\bG_{s,\ne 1}:=\prod_{\langle\lambda\rangle\subset\{\lambda_1,\ldots,\lambda_m\},\lambda\ne 1}\bG_{s,\lambda}.
\end{equation}
By composing $\cL_s^G$ with the natural bijection
\begin{equation} \label{eqn:Ls}
\begin{matrix}
\Irr^\unip(G_s)&\longrightarrow &\Irr^\unip(G_{s,1})\times\Irr^\unip(G_{s,\ne 1})\cr
\pi^\unip&\mapsto&\pi^\unip_{1}\otimes\pi^\unip_{\ne 1}
\end{matrix}
\end{equation}
we obtain a bijection
\begin{equation} \label{mod-Lus-unipotent-decomposition}
\begin{matrix}
   {\tcL}_s^G\colon \mathcal{E}(G,s)&\xrightarrow{1-1}&\Irr^\unip(G_{s,1})\times \Irr^\unip(G_{s,\ne 1}) \cr
    \pi&\mapsto& \pi^\unip_{1}\otimes\pi^\unip_{\ne 1}.
\end{matrix}
\end{equation}
By collecting the bijections $\tcL_s^G$ for all the semisimple element $s$, we obtain a bjijection 
\begin{equation} \label{eqn:JordanUU}
\tcL^G\colon\Irr(G)\longrightarrow \bigcup_{s}\Irr^\unip(G_{s,1})\times \Irr^\unip(G_{s,\ne 1}).
\end{equation}

\subsubsection{Symplectic and orthogonal groups}\

Let $\bG\in\{\SO_{2n+1},\Sp_{2n}, \rO^\pm_{2n}\}$. Then the  group $\bG_s$ decomposes as a product
 \begin{equation}
 \bG_s=\bG_{s,1}\times \bG_{s,-1}\times \bG_{s,\ne},
 \end{equation}
 where
\begin{equation}
\bG_{s,\ne}:=\prod_{\langle\lambda\rangle\subset\{\lambda_1,\ldots,\lambda_m\},\lambda\ne\pm 1}\bG_{s,\lambda}
\end{equation}
is a product of general linear groups and unitary groups.

We write
\begin{equation} \label{decomp-s}
s=s_1\times s_{-1} \times s_{\neq},
\end{equation}
where $s_1$ (resp. $s_{-1}$) is the part whose eigenvalues are all equal to $1$ (resp. $-1$), and $s_{\neq}$ is the part whose 
eigenvalues do not contain $1$ or $-1$.

By composing $\cL_s^G$ with the natural bijection
\begin{equation} \label{eqn:Ls2}
\begin{matrix}
\Irr^\unip(G_s)&\longrightarrow &\Irr^\unip(G_{s,1})\times \Irr^\unip(G_{s,-1})\times\Irr^\unip(G_{s,\ne})\cr
\pi^\unip&\mapsto&\pi^\unip_{1}\otimes\pi^\unip_{-1}\otimes\pi^\unip_{\ne}
\end{matrix}
\end{equation}
we obtain a bijection
\begin{equation}\label{mod-Lus-unipotent-decomposition2}
\begin{matrix}
   {\tcL}_s^G\colon \mathcal{E}(G,s)&\xrightarrow{1-1}&\Irr^\unip(G_{s,1})\times \Irr^\unip(G_{s,-1})\times\Irr^\unip(G_{s,\ne}) \cr
    \pi&\mapsto& \pi^\unip_{1}\otimes\pi^\unip_{-1}\otimes\pi^\unip_{\ne}.
\end{matrix}
\end{equation}
By collecting the bijections $\tcL_s^G$ for all the semisimple element $s$, we obtain a bjijection 
\begin{equation} \label{eqn:JordanSpO}
\tcL^G\colon\Irr(G)\longrightarrow \bigcup_{s}\Irr^\unip(G_{s,1})\times \Irr^\unip(G_{s,-1})\times\Irr^\unip(G_{s,\ne}).
\end{equation}

\begin{notation} \label{notation:Epi}
{\rm
Let $s$ be a semisimple element in $\SO_{2n}^\epsilon$ and let $\pi\in\cE(\rO_{2n}^\epsilon(k),s)$. We attach to $\pi$ the following set of representations:
\[
\cE(\pi):=\left\{ \pi^\unip_{1}\otimes\pi^\unip_{-1}\otimes\pi^\unip_{\ne}, 
\pi^\unip_{1,\sgn}\otimes\pi^\unip_{-1}\otimes\pi^\unip_{\ne},
\pi^\unip_{1}\otimes\pi^\unip_{-1,\sgn}\otimes\pi^\unip_{\ne},
\pi^\unip_{1,\sgn}\otimes\pi^\unip_{-1,\sgn}\otimes\pi^\unip_{\ne}
\right\},
\]
where for  $\pi\in \Irr(\rO_{2n}^\epsilon(k))$, we write $\pi_{\sgn}:=\pi\otimes\sgn$.}
\end{notation}

\subsubsection{The case of $\bG=\rO_{2n+1}$} \

We have $\bG^\circ=\SO_{2n+1}$ and $\rO_{2n+1}\simeq\SO_{2n+1}\times\{\pm 1\}$. For $\pi\in \Irr(\SO_{2n+1}(k))$, we set 
\[\pi^+:=\pi\otimes 1 \quad\text{and}\quad \pi^-:=\pi\otimes\sgn.\]
We obtain

\[\Irr(\rO_{2n+1}(k))=\left\{\pi^+,\pi^-\,:\, \pi\in\Irr(\SO_{2n+1}(k))\right\}.\]
We have $\check\bG^{\circ}=\Sp_{2n}$.
Let $s$ be a semisimple element of $\Sp_{2n}$. We set
\[\cE(\rO_{2n+1},s,\pm):=\left\{\pi^\pm\,:\, \pi\in\cE(\SO_{2n+1},s)\right\}.\]
We have
\begin{equation} \label{eqn:IrrOodd}
\Irr(\rO_{2n+1}(k))=\bigcup_{s\in \Sp_{2n}(k)\atop \zeta\in\{+,-\}}\cE(\rO_{2n+1},s,\zeta).
\end{equation}

\begin{remark} {\rm If $\pi\in \cE(\Sp_{2n},s)$ or if $\pi\in \cE(\rO_{2n+1},s,\zeta)$, then we have $\pi(-1)=\zeta_\pi\pi(1)$ with $\zeta_\pi\in\{+,-\}$.}
\end{remark} 

\subsubsection{Notation} \label{subsubsec:Not}\

For later use, we set
\begin{equation} \label{equation:Gnatural}
\bG_{s,\natural}:=\begin{cases}
\bG_{s,-1}\times \bG_{s,\ne}&\text{if $\check G\in\{\SO_{2n+1},\rO^\epsilon_{2n}\}$,}\cr
\bG_{s,1}\times \bG_{s,\ne}&\text{if  $\check G=\SO_{2n}^\epsilon$,}
\end{cases}
\end{equation}
and we write
\begin{equation}
\pi^\unip_\natural:=\begin{cases}
\pi^\unip_{-1}\otimes\pi^\unip_{\ne}
&\text{if $\check G\in\{\SO_{2m+1},\rO^\pm_{2m}\}$,}\cr
\pi^\unip_{1}\otimes\pi^\unip_{\ne}&\text{if $\check G=\SO_{2m}$,}
\end{cases}
\end{equation}
so that we have
\begin{equation}
\pi^\unip=\begin{cases}
\pi^\unip_1\otimes\pi^\unip_\natural&\text{if $\check G\in\{\SO_{2m+1},\rO^\pm_{2m}\}$,}\cr
\pi^\unip_{-1}\otimes\pi^\unip_{\natural}&\text{if $\check G=\SO_{2m}$.}
\end{cases}
\end{equation}

\subsection{Reduction of the theta correspondence to unipotent representations} \label{sec:reduc}

The following results show that the description of the theta correspondence reduces to describe its restriction to unipotent representations.

\subsubsection{Pairs $(\rU_{m},\rU_{m'})$}
\begin{thm} \label{thm:reducUU}
Let $(\bG,\bG'):=(\rU_{m},\rU_{m'})$, let $(\pi,\pi')\in\cE(G,s)\times\cE(G',s')$ for some semisimple elements $s\in \check G$ and $s'\in \check G'$.

Then  $(\pi,\pi')\in\Theta_{G,G'}$  if and only if the following conditions hold:
\begin{enumerate}
\item $G_{s,\ne 1}\simeq G^{\prime}_{s',\ne 1}$, and $\pi^\unip_{\neq 1}=(\pi')^\unip_{\neq 1}$;
\item $\pi^\unip_1\otimes(\pi')^\unip_1$ occurs in $\Theta_{G_{\vphantom{s'}s,1},G'_{s',1}}$.
\end{enumerate}
\end{thm}
\begin{proof}
See \cite[Th\'eor\`eme~2.6]{AMR}.
\end{proof}

\subsubsection{Pairs $(\Sp_{2n},\rO_{2n'}^\epsilon)$}
\begin{thm}  \label{thm:reducSpOeven}
Let $(\bG,\bG'):=(\Sp_{2n},\rO_{2n'}^\epsilon)$, let $(\pi,\pi')\in\cE(G,s)\times\cE(G',s')$ for some semisimple elements $s\in \SO_{2n+1}(k)$ and $s'\in \SO_{2n'}^\epsilon(k)$,

Then  $(\pi,\rho')\in\Theta_{G,G'}$ for some $\tilde\rho'\in \cE(\pi')$ such that $\tcL_{s'}^{G'}(\rho')=\tilde\rho'$ if and only if the following conditions hold:
\begin{enumerate}
\item $s_{\neq}=s'_{\neq}$ (up to conjugation), and $\pi^\unip_{\neq}=(\pi')^\unip_{\neq}$;
\item $G_{s,-1}\simeq G^{\prime}_{s',-1}$, and $\pi^\unip_{-1}\in\{(\pi^{\prime})^{\unip}_{-1},(\pi')^{\unip}_{-1,\sgn}\}$;
\item $\pi^\unip_1\otimes(\pi')^\unip_1$ or $\pi^\unip_1\otimes(\pi')^\unip_{1,\sgn}$ occurs in $\Theta_{\Sp_{\vphantom{s'}2\nu_1(s)},\rO^\epsilon_{2\nu_1(s')}}$.
\end{enumerate}
\end{thm}
\begin{proof}
See \cite[Theorem~5.10]{Pan-MathAnn}.
\end{proof}

\subsubsection{Pairs $(\Sp_{2n},\rO_{2n'+1}^\epsilon)$}
\begin{thm}  \label{thm:reducSpOodd}
Let $(\bG,\bG'):=(\Sp_{2n},\rO_{2n'+1})$, let $(\pi,\pi')\in\cE(G,s)\times\cE(G',s',\zeta')$ for some semisimple elements $s\in \SO_{2n+1}(k)$ and $s'\in \Sp_{2n'}(k)$,  and some $\zeta'\in\{-,+\}$.
Write $\tcL_{s'}(\pi'|_{\SO_{2n'+1}})=(\pi')^\unip_1\otimes(\pi')^\unip_{-1}\otimes(\pi')^\unip_{\neq}$.

Then  $(\pi,\pi')\in\Theta_{G,G'}$ if and only if the following conditions hold:
\begin{enumerate}
\item $s_{\neq}=-s'_{\neq}$ (up to conjugation), and $\pi^\unip_{\neq}=(\pi')^\unip_{\neq}$;
\item $G_{s,1}=\Sp_{2\nu_1(s)}\simeq G^{\prime}_{s',-1}=\Sp_{2\nu_{-1}(s')}$, and $\pi^\unip_{-1}=(\pi^{\prime})^{\unip}_{1}$;
\item $\pi^\unip_{-1}\otimes(\pi')^\unip_{1}$ or $\pi^{\unip}_{-1,\sgn}\otimes(\pi')^\unip_{1}$ occurs in $\Theta_{\Sp_{\vphantom{s'}2\nu_{-1}(s)},\rO^\epsilon_{2\nu_{1}(s')}}$.
\item $\zeta_{\pi'}=\zeta'$.
\end{enumerate}
\end{thm}
\begin{proof}
See \cite[Theorem~7.6]{Pan-MathAnn}.
\end{proof}

\section{The theta correspondence for unipotent representations} \label{sec:unip}

\subsection{Dual pairs of type II}\

Let $\bG$ be a connected reductive algebraic group defined over $k$, and let $\bT$ be a maximal torus in $\bG$. Let $F$ be the Frobenius endomorphism and $W$ the Weyl group of $\bG$.

\begin{defn} {\rm
The  \textit{$F$-conjugation} in $\bG$ is defined to be the action of $\bG$ on itself defined for any $g\in \bG$ by $x\mapsto gx({}^Fg)^{-1}$.

The $F$-conjugacy class of the image of $g^{-1}{}^Fg$ in $W$ is called the \textit{type} of the torus ${}^g\bT$ with respect to $\bT$; a representative of this 
$F$-conjugacy class is called a type of ${}^g\bT$ with respect to $\bT$.}
\end{defn}

By the conjugation action of $g^{-1}$, the torus ${}^g\bT$, endowed with the action of $F$, is identified with the torus $\bT$ endowed with the action of $wF$, if $w$ is the image of $g^{-1}{}^Fg$ in $W$. 
We will denoted by \textit{$\bT_w$} this torus.

\smallskip

We now take for $\bG$  the general linear group $\GL_n$. The unipotent irreducible representations of $G=\GL_n(k)$ are the irreducible components of $\ii_{\bT,\bB}^\bG (1)$, where $\bB$ is a Borel subgroup of 
$\bG$ and $\bT\subset\bB$ a maximal torus, and $1$ denotes the trivial representation of $T=\bT(k)\simeq k^\times$. 
Hence  the  set $\Irr_{\unip}(G)$ is in bijection with $\Irr(\fS_n)$, where $\fS_n$ is symmetric group. It follows that the unipotent irreducible representations of $G$ are parametrized  the partitions of $n$. 
We denote by $\pi_\lambda$ and $E_\lambda$  the irreducible representation of $G$ and $\fS_n$, respectively, which corresponds to the partition $\lambda$ of $n$.

We consider the \textit{almost character}:
\begin{equation}
 R^G_\lambda:=\frac{1}{n!}\sum_{w\in \fS_n}\Tr(E_\lambda)(w)\,R_{\bT_w}^{\bG}(1).
 \end{equation}

\begin{prop}
For every $\lambda\in\cP(n)$, the almost character $R_{E_\lambda}:=R_\lambda$ is a unipotent irreducible character of $G$, and each such character is of this form for some partition $\lambda$ of $n$.
\end{prop}

Let $\lambda=(\lambda_1\ge\lambda_2\ge\cdots\ge\lambda_h)\in\cP(n)$ and  $\mu=(\mu_1\ge\mu_2\ge\cdots\ge\mu_{h'})\in\cP(m)$. We may assume that $h=h'$ by adding $0$s if necessary.  
We define  the \textit{union} of $\lambda$ and $\mu$, denoted  $\lambda\cup\mu$,  to be the partition of $n+m$ with parts  $\lambda_1,\ldots,\lambda_h,\mu_1,\ldots,\mu_{h}$. 
The usual order on $\cP(n)$ is defined by
\begin{equation}
\lambda\le\lambda'\quad\text{if and only if}\quad \lambda_1+\cdots+\lambda_i\le \lambda_1'+\cdots+\lambda_i',\text{ for all $i\in\NN$.}
\end{equation}

Let $\lambda=[\lambda_1,\ldots,\lambda_h]$ be a partition of $n$. The  \textit{height} of $\lambda$, denoted $\height(\lambda)$,  is defined to be  the largest $i$ with $\lambda_i\ne 0$.

For every index $j$, we set 
\begin{equation}
\lambda_j^\ast:= |\left\{i \,:\, \lambda_i \ge j\right\}|.
\end{equation} 
The partition ${}^t\lambda=[\lambda^\ast_1, \lambda^\ast_2, \ldots]$ is called  the \textit{transpose} of $\lambda$.
Flipping a Young diagram of  $\lambda$ of over its main diagonal (from upper left to lower right), we obtain the Young diagram of the partition ${}^t\lambda$.

If $\lambda=[\lambda_1\ge\lambda_2\ge\cdots\ge\lambda_h]$ and  $\mu=[\mu_1\ge\mu_2\ge\cdots\ge\mu_{h'}]$ are two partition we write $\mu\subset\lambda$ 
if the followings hold: 
\[\text{$\height(\mu)\le \height(\lambda)$ and
$\mu_i\le\lambda_i$  for all $1\le i\le\height(\mu)$.}\]
 If we identify $\lambda$ and $\mu$ with their Young diagrams, this means that the diagram of $\mu$ is contained in those of $\lambda$. Removing the boxes of $\lambda$ which belong to $\mu$, we obtain a skew diagram which we denote by 
 $\lambda-\mu$.
 
We will also need to consider the \textit{intersection partition} of of $\lambda$ and $\mu$:
\[\lambda\cap\mu:=[\min(\lambda_1,\mu_1), \ldots, \min(\lambda_{\min(h,h')}, \mu_{\min(h,h')})] .\]
We have $\mu\subset\lambda$ if and only if $\lambda\cap\mu=\mu$.

Let $\nu:=[\nu_1\ge \nu_2 \ge\cdots\ge \nu_{m}]$. Then we denote by $p_{\lambda=\mu}(\nu)$ the partition $[\nu_i]_{\{i:\lambda_i=\mu_i\}}$ and we put
\begin{equation}
\lambda\cap^{=}\mu:= p_{\lambda=\mu}(\lambda\cap\mu).
\end{equation}
The partitions $\lambda$and $\mu$  are said to be \textit{close} if for each $i$ we have $|\lambda_i -\mu_i| \le 1$.

We consider the free $\ZZ$-module:
\[\fR(\fS):=\bigoplus_{n\ge 0}\fR(\fS_n),\]
and define a map $\theta^{\fS}\colon \fR(\fS)\to \fR(\fS)$ by
\begin{equation}
E_{\lambda}\mapsto \sum_{\lambda'\atop
\text{${}^t\lambda'$ is close to ${}^t\lambda$}} f({}^t\lambda\cap^{=}{}^t\lambda')\,E_{\lambda'},
\end{equation}
where, if $\mu=[r^{a_1}, (r-1)^{a_2},\ldots, 1^{a_r}]$, we have put $f(\mu):=\prod_i a_i$.
\begin{thm}
The theta correspondence between unipotent representations of $G=\GL_n(k)$ and $G'=\GL_{n'}(k)$ is given by the map
\[R^G_\lambda\mapsto R^{G'}_{\theta^{\fS}(E_{\lambda})}.\]
\end{thm}
\begin{proof}
See \cite[Th\'eor\`eme~5.5]{AMR}.
\end{proof}

\begin{cor} \label{cor:AKP2}
We consider the dual pair  $(G,G'):= (\GL_n(k),\GL_{n'}(k))$ with $n'\le n$. The unipotent representations of $G$ and $G'$
with characters $R^G_\lambda$ and $R^{G'}_{ \lambda\cup (n-n')}$, where $\lambda\in\cP(n)$ and $\lambda'\in\cP(n')$, correspond under the theta correspondence.

Moreover, any representation of $G$ which belongs to the image of $R_\lambda$ under the theta correspondence is of the form
\[ \text{$R_{\mu'}$ where $\mu'\ge \lambda\cup (n-n')$}.\]
\end{cor}
\begin{proof}
See \cite[Theorem~3]{AKP2}.
\end{proof}

\subsection{Cuspidal unipotent representations of classical groups} \label{subsec:CuspUnipt}\

Let $G$ be a classical group. It has cuspidal unipotent irreducible representations in the following cases:
\begin{itemize}
\item[$\bullet$]
$G=\GL_1(k)$ has a unique  such representation: the trivial representation;
\item[$\bullet$]
$G\in\{\Sp_{2(c^2+c)}(k),\SO_{2(c^2+c)+1}(k),\rU^\epsilon_{(c^2+c)/2}(k)\}$, with  $c$ a positive integer,  has a unique such representation, say $\pi^G_c$;
\item[$\bullet$]
$G=\SO_{2c^2}^+(k)$, with  $c$ a positive even integer, has a unique such representation, say $\pi^G_c$;
\item[$\bullet$]
$G=\SO_{2c^2}^-(k)$, with $c$ a positive  odd integer, has a unique such representation, say $\pi^G_c$;
\item[$\bullet$]
$G=\rO_{2c^2}^+(k)$,  with $c$  a positive even integer,  has two  such representations, say $\pi^G_{c,+}$ and $\pi^G_{c,-}=\pi^G_{c,+}\otimes\sign$;
\item[$\bullet$]
$G=\rO_{2c^2}^-(k)$, with $c$ a positive odd integer, has two  such representations, say $\pi^G_{c,+}$ and $\pi^G_{c,-}=\pi^G_{c,+}\otimes\sign$;
\item[$\bullet$]
$G=\rO_{2(c^2+c)+1}(k)$, with  $c$ a positive integer, has two  such representations, say $\pi^G_{c,+}$ and $\pi^G_{c,-}=\pi^G_{c,+}\otimes\sign$;
\item[$\bullet$]
$G=\rU^\epsilon_{\frac{c^2+c}{2}}(k)$, with  $c$ a positive integer such that $\epsilon=(-1)^{\frac{c^2+c}{2}}$.
\end{itemize}

\begin{thm} \label{thm:ucuspidals} \

\begin{enumerate}
\item We suppose that  $G=\Sp_{2n}(k)$ and $G'=\rO_{2n'}^\epsilon(k)$. The following holds:
\begin{enumerate}
\item If $\pi$ is a cuspidal irreducible  representation of $G$, then $\theta_{n'_{\cT'}(\pi)}(\pi)$ is a singleton $\{\pi'\}$ with $\pi'$ cuspidal irreducible.
\item If $\pi\in\Irr(G_n)$ is unipotent then any $\pi'\in \theta_{G'}(\pi)$ is unipotent.
\item The representation $\pi_c^G$ of $G=\Sp_{2(c^2+c)}(k)$ corresponds to the representation $\pi^{G'}_{c',-}$ with $c'=c$ if $\epsilon$ is the sign of $(-1)^c$, and to the representation $\pi^{G'}_{c',+}$ with $c'=c+1$ otherwise.
\end{enumerate}
\item We suppose that  $G=\rU_{n}^\epsilon(k)$ and $G'=\rU_{n'}^{\epsilon'}(k)$. The representation $\pi_c^G$ of $G$ corresponds to the representation $\pi^{G'}_{c'}$, where
\begin{enumerate}
\item if $c=0$, then $c'=0$;
\item if $c\ne 0$, then $c'=c\pm 1$ such that $\epsilon'=(-1)^{\frac{c'(c'+1)}{2}}$.
\end{enumerate}
\end{enumerate}
\end{thm}
\begin{proof}
See  \cite{Adams-Moy}.
\end{proof}

Hence, for $(G,G')$ an irreducible reductive dual pair,  the theta  correspondence induces a correspondence between cuspidal unipotent representations of $G$ and $G'$, 
which is described by the function $\theta\colon\NN\to \NN$, defined by 
\begin{equation}
\theta(c):=c',
\end{equation}
with $c'$ as in Theorem~\ref{thm:ucuspidals}.

\subsection{Harish-Chandra series of unipotent representations} \label{subsec:HCunip}\

Let $G=G_n$ be a classical group. We set
\begin{equation} \label{eqn:overline{n}}
\overline{n}:=
\begin{cases}
n-c^2-c&\text{if $G=\Sp_{2n}(k)$}\cr
n-c^2 \text{ (with $\epsilon=(-1)^c$)}&\text{if $G=\rO_{2n}^\epsilon(k)$}\cr
n-c^2-c&\text{if $G=\rO_{2n+1}(k)$}\cr
n-(c^2+c)/2 \text{ (with $\epsilon=(-1)^{\frac{c^2+c}{2}}$)} &\text{if $G=\rU_{n}^\epsilon(k)$}.
\end{cases} 
\end{equation}
If the Levi subgroup $L$ of $G_n$ admits a cuspidal unipotent irreducible representation, then $L$ is of the form
\begin{equation} \label{eqn:Levi-form}
L\simeq
\begin{cases}
\Sp_{2(c^2+c)}(k)\times T_{\overline{n}}&\text{if $G=\Sp_{2n}(k)$}\cr
\rO^\epsilon_{2c^2}(k)\times T_{\overline{n}} \text{ (with $\epsilon=(-1)^c$)}&\text{if $G=\rO_{2n}^\epsilon(k)$}\cr
\rO_{2(c^2+c)}(k)\times T_{\overline{n}}&\text{if $G=\rO_{2n+1}(k)$}\cr
\rU^\epsilon_{(c^2+c)/2}(k)\times T_{\overline{n}}\text{ (with $\epsilon=(-1)^{\frac{c^2+c}{2}}$)} &\text{if $G=\rU^\epsilon_{n}(k)$},
\end{cases} 
\end{equation}
where $T_{\overline{n}}=\bT_{\overline{n}}(k)$, with $\bT_{\overline{n}}$ a maximal torus of $G_{\overline{n}}$.

\begin{prop} \label{pro:principal-unipotent}
For the  irreducible reductive dual pairs 
\[(G,G')\in\{(\Sp_{2n}(k),\rO_{2n'}^+(k)),(\rU_{n}^+(k),\rU_{n}^+(k)),(\GL_n(k),\GL_{n'}(k))\},\] 
and only for these ones,  the images by the theta correspondence of every unipotent representation in the principal series of $G$ belongs to the (unipotent) principal series of $G'$.
\end{prop}
\begin{proof}
It follows from the combination of Theorem~\ref{thm:HC-HC} with Theorem~\ref{thm:ucuspidals}.
\end{proof}

\begin{prop} \label{prop:B}
Let $G=G_n$ be one of the classical groups occurring in \eqref{eqn:overline{n}}, and let $L$ be a Levi subgroup of $G$.
If $\sigma\in\Irr(L)$ is  cuspidal unipotent, then the group 
\begin{equation} \label{eqn:W}
W_{\sigma}:=\rN_{G}(L,\sigma)/L
\end{equation} 
is a finite Weyl group of type $\rB_{\overline{n}}$.
\end{prop}
\begin{proof}
See \cite{LusMadison} for $G$ connected and \cite[Lemme~3.2]{AMR} for $G$ an orthogonal group.
\end{proof}

Let $\fs:=(L,\sigma)_G$ be the $G$-conjugacy class of a pair $(L,\sigma)$, with $L$ a Levi subgroup of $G$ and $\sigma$ a cuspidal irreducible representation of $L$. We denote $\fC(G)$ the set of such $\fs$. 

Let $\fR(G)$ be the category of complex representations of $G$. We define $\fR^\fs(G)$ to be the subcategory of $\fR(G)$ formed by the representations whose irreducible components have cuspidal support $\fs$.

Let $(G,G')$ be an irreducible reductive dual pair, and let $(\fs,\fs')\in\fC(G)\times\fC(G')$. We denote by $\omega_{n,n'}^{\fs,\fs'}$ the projection of the  restriction to $G\cdot G'$ of the Weil representation $\omega_{G,G'}$
(defined in \eqref{eqn:omegaGG}) onto $\fR^\fs(G)\otimes\fR^{\fs'}(G')$.

\begin{thm} \label{thm:HC-unip}
Let $(G,G')=(G_n,G'_{n'})$ be an irreducible reductive dual pair, let $\fs=(L,\sigma)_G$ such that $\sigma$ is unipotent, and let $\pi\in\Irr^\fs(G)$. 
Then we have
\[L\simeq G_{n_0}\times \GL_{1}(k)^{\overline{n}}\quad\text{and}\quad \sigma\simeq \pi_{n_0}\otimes 1^{\overline{n}},\]
with $\pi_{n_0}\in\Irr(G_{n_0})$ cuspidal unipotent,
and, every $\pi'\in\theta_{G'}(\pi)$ belongs to $\Irr^{\fs'}(G')$, where $\fs'=(L',\sigma')_{G'}$ with
\[L'\simeq G'_{n'_0}\times \GL_{1}(k)^{\overline{n}'}\quad\text{and $\sigma'=\theta^0(\pi_{n_0})\otimes 1^{\overline{n}'}$},\]
where $n'_0=n_{\cT'}(\sigma_0)$ and $\overline{n}'=n-n'_0$.
\end{thm}
\begin{proof}
It follows from Corollary~\ref{cor:HC}.
\end{proof}

\begin{cor} \label{cor:Theta-WW}
Let $(G,G')=(G_n,G'_{n'})$ be an irreducible  reductive dual pair. The description of the theta correspondence for unipotent representations of $G$ and $G'$ reduces to the description, for every pair 
$(\fs,\fs')$ as in  Theorem~\ref{thm:HC-unip}, of a correspondence between $\Irr(W_{\overline{n}})$ and $\Irr(W_{\overline{n}'})$, where $W_m$ is a finite Weyl group of type $\rB_m$ for $m\in\{\overline{n},\overline{n}'\}$.
\end{cor}
\begin{proof}
It follows from the combination of Theorem~\ref{thm:HC-unip} with Theorem~\ref{thm:HC:W} and Proposition~\ref{prop:B}.
\end{proof}

\subsection{A correspondence between representations of Weyl groups} \label{subsec:WW}\

Let $m$ be a positive integer. A sequence $\lambda:=(\lambda_1\ge\lambda_2\ge\cdots\ge\lambda_h)$ where $\lambda_i\in\ZZ_{\ge 0}$ is called a partition of $m$ if $\lambda_1+\lambda_2+\cdots+\lambda_h=m$. 
We set $|\lambda|:=\lambda_1+\lambda_2+\cdots+\lambda_h$. The number 
\begin{equation} \label{eqn:length}
\ell(\lambda) := |\{ i | \lambda_i > 0 \}|
\end{equation} 
 is called the \textit{length} of $\lambda$. 
Let $\cP(m)$ denote the set of partitions of $m$. 
A bipartition of $m$ is a pair $\left[\begin{smallmatrix} \lambda\cr \mu\end{smallmatrix}\right]$ of partitions such that $|\lambda|+\mu|=m$.
We denote by $\cP_2(m)$ the set of bipartitions of $m$.

Let  $W_m$ be the finite Weyl group of type $\rB_m$. Let $\sgn\colon W_m\to \{\pm 1\}$ be the unique homomorphism with kernel the subgroup of $W_m$ of type $\rD_m$. The irreducible representations of  $W_m$ are parametrized by the bipartitions of $m$.
We denote by $E_{\lambda,\mu}$ the  irreducible representation of $W_m$ which corresponds to the bipartition $(\lambda,\mu)$ of $m$. 
We keep the notation of Theorem~\ref{thm:HC-unip}.

\subsubsection{The unitary case}\

\begin{thm} \label{thm:unip-UU}
Let $(G,G')=(\rU_n, \rU_{n'})$ be a unitary irreducible reductive dual pair. 
The bijection 
\[\Irr^\fs(G)\times\Irr^{\fs'}(G')\,\overset{1-1}{\longleftrightarrow}\, \Irr(W_{\overline{n}})\times \Irr(W'_{\overline{n}'})\]
identifies the representation $\omega_{n,n'}^{\fs,\fs'}$ with the representation $\Omega_{\overline{n},\overline{n}'}$ defined by:
\begin{enumerate}
\item If $c$ is odd or $c=c'=0$:
\[\Omega_{\overline{n},\overline{n}'}:=\sum_{r=0}^{\min(\overline{n},\overline{n}')}\sum_{(\lambda,\mu)\in\cP_2(r)}
\left(\Ind_{W_r\times W_{\overline{n}-r}}^{W_{\overline{n}}}E_{\lambda,\mu}\otimes 1\right)\,\otimes\,
\left(\Ind_{W_r\times W_{\overline{n}-r}}^{W_{\overline{n}}} \sgn \cdot E_{\lambda,\mu}\otimes1\right),\]
\item otherwise:
\[\Omega_{\overline{n},\overline{n}'}:=\sum_{r=0}^{\min(\overline{n},\overline{n}')}\sum_{(\lambda,\mu)\in\cP_2(r)}
\left(\Ind_{W_r\times W_{\overline{n}-r}}^{W_{\overline{n}}}E_{\lambda,\mu}\otimes\sgn\right)\,\otimes\,
\left(\Ind_{W_r\times W_{\overline{n}-r}}^{W_{\overline{n}}} \sgn \cdot E_{\lambda,\mu}\otimes1\right).\]
\end{enumerate}
\end{thm}
\begin{proof}
See \cite[Th\'eor\`eme~3.10]{AMR}.
\end{proof}

\subsubsection{The  symplectic-orthogonal case}\

In the case of symplectic-orthogonal dual pairs, the following  analogue of Theorem~\ref{thm:unip-UU} has been conjectured in \cite[Conjecture~3.11]{AMR}:
\begin{thm} \label{thm:AMRconj}
Let $(G,G')=(G_n,G_{n'})$ be a symplectic-orthogonal irreducible reductive dual pair. 
The bijection 
\[\Irr^\fs(G)\times\Irr^{\fs'}(G')\,\overset{1-1}{\longleftrightarrow}\, \Irr(W_{\overline{n}})\times \Irr(W'_{\overline{n}'})\]
identifies the representation $\omega_{n,n'}^{\fs,\fs'}$ with the representation $\Omega_{\overline{n},\overline{n}'}$ defined by:
\begin{enumerate}
\item Cases $(\Sp_{2n},\rO^+_{2n'})$ with $c$ even and $(\Sp_{2n},\rO^-_{2n'})$ with $c$ odd: 
\[\Omega_{\overline{n},\overline{n}'}:=\sum_{r=0}^{\min(\overline{n},\overline{n}')}\sum_{(\lambda,\mu)\in\cP_2(r)}
\left(\Ind_{W_r\times W_{\overline{n}-r}}^{W_{\overline{n}}}E_{\lambda,\mu}\otimes\sgn\right)\,\otimes\,
\left(\Ind_{W_r\times W_{\overline{n}-r}}^{W_{\overline{n}}} E_{\lambda,\mu}\otimes\sgn\right),\]
\item Cases $(\Sp_{2n},\rO^+_{2n'})$ with $c$ odd and $(\Sp_{2n},\rO^+_{2n'})$ with $c$ even: 
\[\Omega_{\overline{n},\overline{n}'}:=\sum_{r=0}^{\min(\overline{n},\overline{n}')}\sum_{(\lambda,\mu)\in\cP_2(r)}
\left(\Ind_{W_r\times W_{\overline{n}-r}}^{W_{\overline{n}}}E_{\lambda,\mu}\otimes 1\right)\,\otimes\,
\left(\Ind_{W_r\times W_{\overline{n}-r}}^{W_{\overline{n}}}E_{\lambda,\mu}\otimes \sgn \right).\]
\end{enumerate} 
\end{thm}
\begin{proof}
See \cite[Theorem 1.8 and Proposition 3.9]{PanAJM} or \cite[Theorem 1.4]{MQZ}.
\end{proof}

\section{Combinatorial description} \label{sec:combi}

\subsection{Unitary groups} \label{sec:unitary}\

We  define another order on partitions as follows: for $\lambda$, $\lambda'$ partitions of possibly different integers, we write
\begin{equation}\lambda\preceq\lambda'\quad\text{if and only if}\quad \lambda'_{i+1}\le\lambda_i\le\lambda_i' ,\text{ for all $i\in\NN$.}
\end{equation}
In other words, $\lambda\preceq\lambda'$ if the Young diagram of $\lambda$ is contained in the one of $\lambda'$ and that we can go from the first to the second by adding at most one box per column. 

\begin{example} {\rm 
The partitions $\lambda= (4,1,1)$ of $6$, and $\lambda'= (4,4,1,1)$ of $10$ verify $\lambda\preceq\lambda'$. }
\end{example}

\begin{remark} {\rm Two partitions $\lambda=(\lambda_1\ge\lambda_2\ge\cdots\ge\lambda_h)$  and $\lambda'=(\lambda_1\ge\lambda_2\ge\cdots\ge\lambda_{h})$  satisfy $\lambda\preceq\lambda'$ and $|\lambda|=|\lambda'|$, then we have $\lambda=\lambda'$. }
\end{remark}

The \textit{Young diagram} of  a partition  $\lambda$ is defined to be the set of points $(i, j) \in\NN \times\NN$ such that $1 \le j \le \lambda_i$.
To each point $(i, j)$ in the Young diagram of$\lambda$ we associate the \textit{hook} at $(i,j)$, which the set
\begin{equation}
\xi=\xi_{i,j} :=\left\{(i,j')\,:\, j\le j'\le  \lambda_i\right\}\,\cup\,\left\{(i',j)\,:\,i\le i'\le \lambda^\ast_i \right\},
\end{equation}  
The cardinality 
$|\xi|=\lambda_i +\lambda_j^\ast -i-j+1$
is called the \textit{hook-length} of $\xi$. A hook of length $2$ is also  a \textit{$2$-hook}. A hook 
$\xi_{i',j'}$ of $\lambda$ is said to be 
\begin{itemize}
\item \textit{above} $\xi_{i,j}$ if we have $j'=j$ and $i'<i$
\item \textit{left to}  $\xi_{i,j}$  if $i'=i$ and $j'<j$.
\end{itemize}
Let  $\lambda=[\lambda_1,\ldots,\lambda_m]$ and $\lambda'=[\lambda'_1,\ldots,\lambda'_m]$ be
two partitions of $n$ such that either
\begin{itemize}
\item there exists an index $i_0$ such that $\lambda_{i_0}'=\lambda_{i_0}+2$, and $\lambda_i'=\lambda_i$ for $i\ne i_0$;
\item
or  there exists an index $i_0$ such that $\lambda_{i_0}'=\lambda_{i_0}+1$,  $\lambda_{i_0+1}'=\lambda_{i_0+1}+1$, and
$\lambda_i'=\lambda_i$ for $i\ne i_0,i_0+1 $,
\end{itemize}
then we said that $\lambda'$ is obtained from $\lambda$ by adding a $2$-hook or $\lambda$ is obtained from $\lambda'$ by removing a $2$-hook.

Let $\lambda$ be a partition of $n$. After removing all possible $2$-hooks step by step, the resulting partition is denoted by $\lambda_\infty$ and called the \textit{$2$-core} of $\lambda$. 
We observe that $\lambda_\infty= [d,d-1, \ldots, 1]$ for some non-negative integer $d$. A partition $\lambda$  is said to be \textit{cuspidal} if $\lambda=\lambda_\infty$.

\begin{defn} A $\beta$-set $X = \left\{x_1,\ldots,x_m\right\}$ is a finite set of non-negative integers with elements written in increasing order, i.e., $x_1 < x_2 <\cdots<  x_m$. We define an equivalence relation (denoted by “$\sim$”) on $\beta$-sets generated by
\[X \sim \{0\}\,\cup\,\left\{ x + 1 |\,:\,x \in X\right\} .\]
\end{defn}
The \textit{rank} of $X$ is 
\[\rank(X): = \sum_{x\in X}x-\binom{|X|}{2}.\]
We have $\rank(X')=\rank(X)$ if $X'\sim X$, and the mapping
\begin{equation}
 \left\{x_1, x_2 ,\ldots,x_m\right\} \mapsto [x_m,x_{m-1}-1,\ldots, x_2-(m-2),x_1- (m -1)]
\end{equation}
defines a bijection from set of equivalence classes of $\beta$-sets of rank $n$ onto the set $\cP(n)$.

Conversely, we associate to every partition $\lambda$  the  $\beta$-set $X_\lambda$ defined by
\begin{equation} \label{eqn:b-set}
X_\lambda:=
\begin{cases}
\{\lambda_h,\lambda_{h-1}+1,\ldots,\lambda_2+(h-2),\lambda_1+h\}, &\text{if $\ell(\lambda)+\ell(\lambda_\infty)$ is even;}\cr
\{0,\lambda_h+1,\lambda_{h-1}+2,\ldots,\lambda_2+(h-1), \lambda_1+h\}, &\text{otherwise.}
\end{cases}
\end{equation}
It is clear that the rank of $X_\lambda$ is equal to $|\lambda|$. Moreover, $X_{\lambda'}\sim X_\lambda$ if $\lambda'$ is obtained
from $\lambda$ by adding several $0$’s.

We set
\begin{equation}
X_\lambda^0:=\{x\in X_\lambda\,:\,x\equiv 0 \text{ ($\modu 2$)}\}
\quad\text{and}\quad
X_\lambda^1:=\{x\in X_\lambda\,:\,x\equiv 1 \text{ ($\modu 2$)}\}  .
\end{equation}

\begin{defn} {\rm 
A \textit{symbol} is an array $\Lambda$ of the form 
\[ \Lambda=\left(\begin{matrix}  A\cr B\end{matrix}\right)=\left(\begin{matrix} a_1&a_2&\ldots&a_{r}\cr 
b_1&b_2&\ldots&b_{s}\end{matrix}\right)\]
of non-negative integers satisfying
 $a_i<a_{i+1}$ and $b_i<b_{i+1}$ for every $i$.}
\end{defn}
The \textit{rank}  of $\Lambda$ is defined to be
\[\rank(\Lambda):=\sum_{i=1}^r a_i+ \sum_{i=1}^s b_i-\left\lfloor\left(\frac{r+s-1}{2}\right)^2\right\rfloor.\]
The \textit{defect} of $\Lambda$ is defined as
\[\defect(\Lambda):=r-s.\]
A symbol $\Lambda$ of defect $1$  is said to be \textit{distinguished} if
\[a_1\le b_1\le a_2\le b_2\le\cdots\le a_{r-1}\le b_{r-1}\le a_{r}.\]
Similarly, a symbol $\Lambda$ of defect $0$ is said to be \textit{distinguished} if
\[a_1\le b_1\le a_2\le b_2\le\cdots\le b_{r-1}\le a_r\le b_r.\]
There is an equivalence relation on symbols generated by the shift
\begin{equation}
\left(\begin{matrix}  A\cr B\end{matrix}\right) \sim \left(\begin{matrix} \{0\}\cup\{A+1\}\cr  \{0\}\cup\{B+1\}\end{matrix}\right).
\end{equation}
The functions $\rank(\cdot)$ and $\defect(\cdot)$ are invariant under the shift operation, hence are well-defined on the set
of symbol classes.
We denote by $\cS_{n}$ the set of symbol classes of rank $n$ and by $\cS_{n,d}$ the set of symbols classes of rank $n$ and defect $d$.
Each symbol class contains a unique symbol $\Lambda$ such that $0\notin A\cap B$. Such a symbol is called \textit{reduced}. From now on, unless otherwise specified, a symbol will always assumed to be reduced.

For a symbol $\Lambda=\binom{A}{B}$, we denote by $\Lambda^\ast$ (resp.~$\Lambda_\ast$) the first row (resp.~second row)
of $\Lambda$, that is,  $\Lambda^\ast:=A$ and $\Lambda_\ast:=B$. 

We associate to $\lambda\in\cP(n)$  the symbol  $\Lambda_\lambda$ defined by
\begin{equation} \label{eqn:symbol-part}
\Lambda_\lambda:=
\begin{cases}
\left(\begin{smallmatrix}X^1_\lambda\cr
X^0_\lambda\end{smallmatrix}\right) & \text{if $\ell(\lambda_\infty)$ is even},\\
\\
\left(\begin{smallmatrix}X^0_\lambda\cr
X^1_\lambda\end{smallmatrix}\right) & \text{if $\ell(\lambda_\infty)$ is odd}.
\end{cases}
\end{equation}
By \cite[Lemma~2.6]{PanUU}, we have
\begin{equation}
\defect(\Lambda_\lambda)=\begin{cases}
\phantom{-}\ell(\lambda_\infty)& \text{if $\ell(\lambda_\infty)$ is even};\\
-\ell(\lambda_\infty)& \text{if $\ell(\lambda_\infty)$ is odd}.
\end{cases}
\end{equation}

Let $\Upsilon \colon \cS_n\to\cP_2(n)$ be the map defined by
\begin{equation} \label{eqn:Upsilon}
\left(\begin{matrix} a_1&a_2&\ldots&a_{r}\cr 
b_1&b_2&\ldots&b_{s}\end{matrix}\right)\mapsto
\left[\begin{matrix} a_{r}-(r-1)&a_{r-1}-(r-2)&\ldots&a_2-1&a_1\cr 
b_{s}-(s-1)&b_{s-1}-(s-2)&\ldots &b_2-1&b_1\end{matrix}\right].
\end{equation}
We set
\begin{align*}
\cS^+ &=\{\,(\Lambda,\Lambda')\in\cS\times\cS
\mid\Upsilon(\Lambda_*)^t\preccurlyeq\Upsilon(\Lambda'^*)^t,\ \Upsilon(\Lambda'_*)^t \preccurlyeq\Upsilon(\Lambda^*)^t\,\};\\
\cB^- &=\{\,(\Lambda,\Lambda')\in\cS\times\cS
\mid\Upsilon(\Lambda^*)^t\preccurlyeq\Upsilon(\Lambda'_*)^t,\ \Upsilon(\Lambda'^*)^t\preccurlyeq\Upsilon(\Lambda_*)^t\,\};\\
\cB^+_{\rU,\rU} &=\left\{\,(\Lambda,\Lambda')\in\cB^+\mid
{\rm def}(\Lambda')=\begin{cases}
0& \text{if ${\rm def}(\Lambda)=0$};\\
-{\rm def}(\Lambda)+1 & \text{if ${\rm def}(\Lambda)\neq 0$}
\end{cases}\,\right\};\\
\cB^-_{\rU,\rU} &=\left\{\,(\Lambda,\Lambda')\in\cB^-\mid
{\rm def}(\Lambda')=-{\rm def}(\Lambda)-1\,\right\},\\
\cB_{\rU_{\vphantom{n'}n},\rU_{n'}} &=
\begin{cases}
\{(\Lambda_\lambda,\Lambda_{\lambda'})\in\cB^+_{\rU,\rU}\mid |\lambda|=n,\ |\lambda'|=n'\,\}, & \text{if $n+n'$ is even};\\
\{(\Lambda_\lambda,\Lambda_{\lambda'})\in\cB^-_{\rU,\rU}\mid|\lambda|=n,\ |\lambda'|=n'\,\}, & \text{if $n+n'$ is odd}.
\end{cases}
\end{align*}

For $(G,G')$ a reductive dual pair,  the unipotent part of  the restriction to $G\cdot G'$ of the Weil representation $\omega_{G,G'}$, which was defined in \eqref{eqn:omegaGG}, is
\begin{equation}
\omega_{n,n'}^{\unip}:=\sum_{\pi\in\Irr^{\unip}(G)\atop\pi'\in\Irr^{\unip}(G')}\mult_{\pi,\pi'}\,\pi\otimes\pi'.
\end{equation}

Let $F'(x): =F({}^tx^{-1})$, where $F$ is the split Frobenius endomorphism which maps the matrix entries to their $q$th power and ${}^tg$ denotes the transposed matrix of $g$. 
The group $\GL_n^{F'}$ is the group of unitary transformations of $(\FF_{q^2})^n$. The group $\GL_n$ with the Frobenius endomorphism $F'$ is called the unitary group, denoted by $\rU_n$.
The irreducible unipotent representations of $\rU_n(k)$ are in bijection with $\cP(n)$ (see \cite[Corollary~2.4]{LS} or \cite[Theorem~11.7.2]{DMbook}). We denote by $\pi_\lambda$ the  irreducible unipotent representation which is associated to $\lambda\in\cP(n)$.
We recall the description of $\pi_\lambda$:
\begin{equation}
\pi_\lambda=\pm \sum_{\chi\in\Irr(\fS_n)}\chi(ww_0)\,R_\chi,
\end{equation}
where $w_0\in\fS_n$  is the permutation $(1, 2, \ldots, n) \mapsto (n, n -1, \ldots,1)$, and $R_\chi$ is defined in \eqref{eqn:Rf}.

The following proposition is reformulation of \cite[Th\'eor\`eme~3.10]{AMR}.
\begin{prop}
Let $(\rU_n(k),\rU_{n'}(k))$ be a unitary irreducible reductive dual pair. 

We have \[
\omega_{n,n'}^{\unip}=\sum_{(\Lambda_\lambda,\Lambda_{\lambda'})\in\cB_{\rU_{n},\rU_{n'}}}\pi_\lambda\otimes\pi_{\lambda'}.
\]
\end{prop}
\begin{proof}
See \cite[Proposition 3.1]{PanUU}.
\end{proof}

\subsection{Symplectic and orthogonal groups}\

A symbol  $\Lambda$ is said to be \textit{degenerate} if $\Lambda^\ast=\Lambda_\ast$. Clearly, degenerate symbols necessarily have defect $0$. Let $\widetilde\cS_{n,0}$ denote the set
of symbol classes of rank $n$ and defect $0$, in which each degenerate symbol class is 
repeated twice.
A symbol  $\Lambda'$ is called a \textit{subsymbol} of $\Lambda$, denoted by $\Lambda'\subset\Lambda$, if we have $(\Lambda')^\ast\subset \Lambda^\ast$ and $(\Lambda')_\ast\subset \Lambda_\ast$.
If $\Lambda'\subset\Lambda$, we define the \textit{symbol substraction} to be
\begin{equation}
\Lambda\backslash\Lambda':=\left(\begin{matrix}  \Lambda^\ast\backslash(\Lambda')^\ast\cr \Lambda_\ast\backslash(\Lambda')_\ast\end{matrix}\right).
\end{equation}
For two symbols $\Lambda$ and $\Lambda'$, we define  their \textit{union} and \textit{intersection} by
\begin{equation}
\Lambda\cup\Lambda':=\left(\begin{matrix}  \Lambda^\ast\cup(\Lambda')^\ast\cr
\Lambda_\ast\cup(\Lambda')_\ast\end{matrix}\right)\quad\text{and}\quad
\Lambda\cap\Lambda':=\left(\begin{matrix}  \Lambda^\ast\cap(\Lambda')^\ast\cr
\Lambda_\ast\cap(\Lambda')_\ast\end{matrix}\right).
\end{equation}
For a distinguished symbol $Z$ , we define its \textit{subsymbol of singles} $Z_I$ as
\begin{equation} \label{eqn:ZI}
Z_I:=Z\backslash\left(\begin{matrix} Z^\ast\cap Z_\ast\cr
Z^\ast\cap Z_\ast\cr\end{matrix}\right).
\end{equation}
For a symbol $\Lambda$, we set $\Lambda^t:=\left(\begin{matrix} \Lambda_\ast\cr  \Lambda^\ast\end{matrix}\right)$.
We have 
\[\rank(\Lambda^t)=\rank(\Lambda) \quad\text{and}\quad \defect(\Lambda^t)=-\defect(\Lambda).\]

For $M$ a subsymbol of $Z_I$, we set
\begin{equation}
\Lambda_M:=(Z\backslash M)\cup M^t,
\end{equation}
that is, $\Lambda_M$  is the symbol obtained from $Z$ by switching the row position of entries in $M$ and
keeping other entries unchanged. Note that $\Lambda_\emptyset=Z$ and $\Lambda_{Z_I}=Z^t$.

If $Z$ is a distinguished symbol of rank $n$ and defect $1$, we define 
\begin{equation}
\cS_Z^{\Sp_{2n}}:=\left\{\Lambda_M\,:\, \text{$M\subset Z_I$ and $|M|$ even}\right\}.
\end{equation}
We set
\begin{equation}
\cS^{\Sp_{2n}}:=\bigcup_{\text{$Z$ distinguished}\atop\rank(Z)=n,\defect(Z)=1}
\cS_Z^{\Sp_{2n}}.
\end{equation}
We have 
\begin{equation}
\cS^{\Sp_{2n}}=\left\{\Lambda\in\cS_n\,:\,\defect(\Lambda)\equiv 1\;(\modu \,4)\right\}.
\end{equation}

If $Z$ is a distinguished symbol of rank $n$ and defect $0$, we define 
\begin{equation}
\begin{matrix}
\cS_Z^{\rO^+_n}&:=&\left\{\Lambda_M\,:\, \text{$M\subset Z_I$ and $|M|$ even}\right\}\cr
\cS_Z^{\rO^-_n}&:=&\left\{\Lambda_M\,:\, \text{$M\subset Z_I$ and $|M|$ odd}\right\}.
\end{matrix}
\end{equation}
We set
\begin{equation}
\cS^{\rO_{2n}^\epsilon}:=\bigcup_{\text{$Z$ distinguished}\atop\rank(Z)=n,\defect(Z)=0} \cS_Z^{\rO^\epsilon_{2n}}.
\end{equation}
We have 
\[
\cS^{\rO_{2n}^+}=\left\{\Lambda\in\cS_n\,:\,\defect(\Lambda)\equiv 0\;(\modu \,4)\right\}\quad\text{and}\quad
\cS^{\rO_{2n}^-}=\left\{\Lambda\in\cS_n\,:\,\defect(\Lambda)\equiv 2\;(\modu \,4)\right\}.
\]
Let $G\in\{\Sp_{2n}(k),  \rO_{m}^\epsilon\}$. We define an addition
\begin{equation}
\Lambda_{M}+\Lambda_{M'}:=\Lambda_N\quad\text{where $N:=M\cup M'\backslash M\cap M'$.}
\end{equation}
We have $\Lambda+Z=\Lambda$ and $\Lambda+\Lambda=Z$ for every $\Lambda\in\cS_Z^G$. This gives $\cS_Z^G$ a vector space
structure over the field $\FF_2$ with two elements.
On the other hand, if $\Lambda\in\cS^{\rO_{2n}^+}$ and $\Lambda'\in\cS^{\rO_{2n}^-}$, then $\Lambda+\Lambda'\in\cS^{\rO_{2n}^-}$, and, if  $\Lambda\in\cS^{\rO_{2n}^-}$ and $\Lambda'\in\cS^{\rO_{2n}^-} $, 
then $\Lambda+\Lambda'\in\cS^{\rO_{2n}^+}$.

\begin{thm} \label{thm:Lusz-Class}
Let $G\in\{\Sp_{2n}(k), \SO_{2n+1}(k),\SO^\epsilon_{2n}, \rO_{m}^\epsilon\}$.
There is a bijection
\begin{equation} \label{eqn:Lus77Sp}
\begin{matrix}\cS^{G}&\longrightarrow&\Irr^\unip(G)\cr
\Lambda&\mapsto& \pi_\Lambda.
\end{matrix}
\end{equation}
\end{thm}
\begin{proof}
For $G\in\{\Sp_{2n}(k), \SO_{2n+1}(k),\SO^\epsilon_{2n}\}$, the bijection is established in \cite[Theorem~8.2]{LusInvent77}.
\end{proof}

\begin{remark}{\rm 
If $G=\rO_{2n}^\epsilon$, we have $\pi_{\Lambda^t}=\sgn\cdot\pi_\Lambda$.}
\end{remark}

\begin{defn} \label{defn:symbol-order} {\rm
The \textit{order} $\ord(\Lambda)$ of the symbol $\Lambda\in\cS^{G}$ is
\[\ord(\Lambda):=\deg(\pi_\Lambda(1)).\]
}
\end{defn}

\begin{example} {\rm 
In the parametrization from Theorem~\ref{thm:Lusz-Class} the unique cuspidal unipotent irreducible representation $\pi_c^G$ of $G=\Sp_{2(c^2+c)}(k)$  is the representation $\pi_{\Lambda_{\Sp,c}}$, where
\[\Lambda_{\Sp,c}:=\begin{cases}
\left(\begin{smallmatrix} 0&1&\cdots&2c-1&2c\cr
&&-&&\end{smallmatrix}\right)&\text{if $c$ even,}\cr
\left(\begin{smallmatrix} &&-&&\cr
0&1&\cdots&2c-1&2c
\end{smallmatrix}\right)&\text{if $c$ odd}.
\end{cases}\]
We have $\defect(\Lambda_{\Sp,c})=(-1)^c(2c+1)$.
}
\end{example}

\begin{example} {\rm The two cuspidal unipotent irreducible representations $\pi_{c,+}^G$ and $\pi_{c,-}^G$ of $G=\rO^\epsilon_{2(c^2)}(k)$  are the representation $\pi_{\Lambda_{\SO,c}}$ and $\pi_{\Lambda^t_{\SO,c}}$, 
 where
\[\Lambda_{\SO,c}:=\begin{cases}
\left(\begin{smallmatrix} 0&1&\cdots&2c-2&2c-1\cr
&&-&&\end{smallmatrix}\right)&\text{if $c$ even,}\cr
\left(\begin{smallmatrix} &&-&&\cr
0&1&\cdots&2c-2&2c-1
\end{smallmatrix}\right)&\text{if $c$ odd}.
\end{cases}\]
We have $\defect(\Lambda_{\SO,c})=(-1)^c2c$.
}
\end{example}

\subsection{From symbols to bipartitions}\

Let  $\Upsilon_{n,d}$ denote the restriction to $\cS_{n,d}$ of the map $\Upsilon$ defined in \eqref{eqn:Upsilon}. We have
\begin{equation}
\Upsilon_{n,d}\colon\cS_{n,d}\;\overset{1-1}{\longrightarrow}\;
\begin{cases}
\cP_2(n-\frac{d^2-1}{4})&\text{if $d$ odd}\cr
\cP_2(n-\frac{d^2}{4})&\text{if $d$ even}.
\end{cases}
\end{equation}
In particular, we have bijections:
\begin{equation}
\Upsilon_{n,1}\colon\cS_{n,1}\overset{1-1}{\longrightarrow}\cP_2(n) 
\quad\text{and}\quad
\Upsilon_{n,0}\colon\cS_{n,0}\overset{1-1}{\longrightarrow}\cP_2(n).
\end{equation}

\begin{example} \label{ex:Sp6} {\rm
The group $\Sp_6(k)$ has $12$ unipotent irreducible representations, which are distributed in $6$ families as follows:
\begin{itemize} 
\item $4$ singleton families:
$\left\{\binom{3}{-}\right\}$, $\left\{\binom{0\; 2}{2}\right\}$, $\left\{\binom{1\; 2}{1}\right\}$, $\left\{\binom{0\; 1\; 2\; 3}{1\; 2\; 3}\right\}$, 
\item
$2$ families of $4$ elements:
$\left\{\binom{0\;3}{1},\binom{1\; 0}{3},\binom{1\; 3}{0},\binom{0\; 1\; 3}{-}\right\}$ and $\left\{\binom{0\; 1\; 3}{1\; 2},\binom{0\; 1\; 2}{1\; 3},\binom{1\; 2\; 3}{0\; 1},\binom{0\; 1\; 2\; 3}{1}\right\}$.
\end{itemize}
}
\end{example}
\begin{example} \label{ex:O2} {\rm
The group $\rO^+_2(k)$  has $2$ unipotent irreducible representations, which belong to the family:
$\left\{\binom{1}{0},\binom{0}{1}\right\}$.
The group $\rO^-_2(k)$ has $2$ unipotent irreducible representations which belong to the family:
$\bigl\{\binom{-}{0\; 1},\binom{0\; 1}{-}\bigr\}$.
}
\end{example}

\subsection{Explicit description of the correspondence for unipotent representations}\

We write
\begin{equation} \cS_{n,n'}^{\epsilon}:=\cS^{\Sp_{2n}}\times\cS^{\rO^\epsilon_{2n'}}.
\end{equation}
For $(\Lambda,\Lambda')\in\cS_{n,n'}^\epsilon$, we set 
$\left[\begin{smallmatrix} \lambda\vphantom{\lambda'}\cr \mu\vphantom{\mu'}\end{smallmatrix}\right]:=\Upsilon(\Lambda)$ and
$\left[\begin{smallmatrix} \lambda'\cr \mu'\end{smallmatrix}\right]:=\Upsilon(\Lambda')$.
We define
\[
\cB(\Sp_{2n},\rO^+_{2n'}):=\left\{(\Lambda,\Lambda')\in\cS_{n,n'}^+\,:\,\mu'\preceq \lambda, \mu\preceq \lambda'\;\text{ with $\defect(\Lambda')=-\defect(\Lambda)+1$}\right\}
\]
\[
\cB(\Sp_{2n},\rO^-_{2n'}):=\left\{(\Lambda,\Lambda')\in\cS_{n,n'}^-\,:\, \lambda'\preceq \mu, \lambda\preceq \mu'\;\text{ with $\defect(\Lambda')=-\defect(\Lambda)-1$}\right\}.
\]

\begin{thm} \label{thm:main}
Let $(G,G')=(\Sp_{2n},\rO^\epsilon_{2n'})$. We have
\[\omega_{G,G'}^\unip=\sum_{(\Lambda,\Lambda')\in\cB(\Sp_{2n},\rO^\epsilon_{2n'})}\pi_\Lambda\otimes\pi_{\Lambda‘}.\]
\end{thm}
\begin{proof}
See \cite[Theorem~1.8]{PanAJM}.
\end{proof}

Theorem~\ref{thm:main} establishes the validity of \cite[Conjecture~3.11]{AMR} (see \cite[Proposition~3.9]{PanAJM}).

\smallskip

By restricting Theorem~\ref{thm:main} to the case of principal series (see also Proposition~\ref{pro:principal-unipotent}), we immediately obtain the following result.
\begin{cor} \label{cor:ppal-unip}
Let $(G,G')=(\Sp_{2n},\rO^\epsilon_{2n'})$ and let $\omega_{G,G'}^{\unip,\prin}$ denote the restriction of $\omega_{G,G'}$ to the principal unipotent series of $G$ and $G'$.
We have
\[\omega_{G,G'}^{\unip,\prin}=\sum_{(\Lambda,\Lambda')\in\cB(\Sp_{2n},\rO^\epsilon_{2n'})^{1,0}}\pi_\Lambda\otimes\pi_{\Lambda‘},\]
where {\rm
\[\cB(\Sp_{2n},\rO^\epsilon_{2n'})^{1,0}:=\left\{(\Lambda,\Lambda')\in\cB(\Sp_{2n},\rO^\epsilon_{2n'})\,:\,\text{$\defect(\Lambda)=1$ and $\defect(\Lambda')=0$}\right\}.\]}
\end{cor}

\section{One-to-one correspondences} \label{sec:one-to-one}\

We denote by $\Theta$ the mapping given by
\begin{equation}
\Theta\colon (G,G')\mapsto \Theta_{G,G'}\subset \Irr(G)\times\Irr(G'),
\end{equation}
where $\Theta_{G,G'}$ is defined in \eqref{eqn:thetaGG}.
If $\widetilde{\Theta}$ is another mapping $\widetilde{\Theta}\colon (G,G')\mapsto \widetilde{\Theta}_{G,G'}$ such that $\widetilde{\Theta}_{G,G'}$ is a subset of $\Theta_{G,G'}$ for every reductive dual pair $(G,G')$, then 
$\widetilde{\Theta}$ is called a \textit{subrelation} of $\Theta$.

For every subrelation $\widetilde{\Theta}$ of $\Theta$ and $\pi\in\Irr(G)$, we define
\begin{equation}
\widetilde\Theta_{G'}(\pi):=\left\{\pi'\in \Irr(G')\,:\,(\pi,\pi')\in \widetilde\Theta_{G,G'}\right\}.
\end{equation}

\begin{defn} {\rm 
A subrelation $\widetilde{\Theta}$ of $\Theta$ is \textit{one-to-one} if for every reductive dual pair $(G,G')$ and every $\pi\in\Irr(G)$, there exists at most one $\pi'\in\Irr(G)$  
such that $(\pi,\pi')\in\widetilde{\Theta}_{G,G'}$.}
\end{defn}

\subsection{The eta correspondence}\

Gurevich and Howe  proposed a candidate for a one-to-one sub-relation,  that they called the \textit{eta correspondence},  for  the dual pair $(G,G')=(\rO^\epsilon_m(k),\Sp_{2n'}(k))$, 
where $\epsilon=\pm$ and $m\le n'$, \textit{i.e.}, the dual pair $(G,G')$ is in stable range, as part of a new approach for the study of
representations of classical groups over finite and local fields described in \cite{How-rank,GH,GH2,How2}. 

Let $V$ be a $2n$-dimensional symplectic vector space over $k$, and  $V=X\oplus Y$ be a polarization of $V$. 
We denote by $\Sym^2(Y)$ the space of symmetric bilinear forms on $X\simeq Y^*$, and by $L\rtimes U$ the Levi decomposition of the Siegel parabolic subgroup of  $\Sp_{2n}(k)=\Sp(V)$:
\[L\simeq \GL(X) \quad\text{and}\quad U\simeq\Sym^2(X).\]
We set
\[\text{$\psi_B(A):=\psi(\Tr(BA))$ for every $A\in\Sym^2(X)$},\]
where $Y\overset{A}{\to} X\overset{B}{\to}Y$.
If $B$ is a form of rank $r$,  we  say that the associated character $\psi_B$ has rank $r$.

The following notion of \textit{rank}  for representations  of finite symplectic groups was introduced in \cite{GH}.
Take a representation $\pi$ of $\Sp_{2n}(k)$. The  restriction  of $\pi$ to $U$ decomposes
as a sum of characters with certain multiplicities
\[ \pi|_{N}=\sum_{B\in\Sym^2(Y)} m_B\,\psi_B.\]
Then  $\pi$ is  said to have rank $r$ if $\pi|_{N}$ contains representations of rank $r$, but of no higher rank.

\begin{thm} 
Let $G':=\rO_{2n'}(k)=\rO(V')$. 
We suppose that $n'<n$. 
\textit{
For every $\pi\in\Irr(G')$, the representation $\theta_{G',G}(\pi')$ contains a unique
irreducible constituent $\eta(\pi)$
of rank $r$, and all other constituents have rank less
than $r$
}.
\end{thm}
\begin{proof}
See \cite[Theorem~3.3.3]{GH}.
\end{proof}

\subsection{Two one-to-one extensions of the eta correspondence}\

For the dual pair $(G,G')=(\rO_{2n}^\epsilon(k),\Sp_{2n'}(k))$ or $(\Sp_{2n}(k),\rO_{2n'}^\epsilon(k)(k))$, the theta correspondence induces a correspondence between $\Irr^\unip(G)$ and $\Irr^\unip(G')$.
Following \cite{AKP2} and \cite{Epequin}, we consider  the map
$\underline{\Theta}_{G'}\colon\Irr^\unip(G)\longrightarrow \Irr^\unip(G')$ defined by $\underline{\Theta}_{G'}(\pi_\Lambda):=\pi_{\Lambda'}$, where $\Lambda'\in\cS^{G'}$ is the symbol such that (for a suitable $r$)
\begin{equation}
\Upsilon(\Lambda'):=\begin{cases}
\Upsilon(\Lambda)^t\cup \left[\begin{smallmatrix}r\cr -\end{smallmatrix}\right]&\text{if $\epsilon=+$},\cr
\Upsilon(\Lambda)^t\cup \left[\begin{smallmatrix}-\cr r\end{smallmatrix}\right] ,&\text{if $\epsilon=-$}.
\end{cases}
\end{equation}
Then  $\underline{\Theta}$ is a subrelation of  the restriction of $\Theta_{G,G'}$ to unipotent representations by Theorem~\ref{thm:main}. It also clearly follows from its definition that $\underline{\Theta}$ is one-to-one.
However, the degree in $q$ of $\pi_{\Lambda'}(1)$ is not necessary maximal  in $\Theta_{G'}(\pi_\Lambda)$. 

Another one-to-one subrelation $\overline{\Theta}$ has been defined by Pan in \cite{Pan-eta-SpO, PanUU}, and extended to arbitrary irreducible representations. Recall that $\cS_{n,d}\subset \cS_G$ denotes 
the set of symbols of rank $n$ and defect $d$. Let ``$<$" be the linear order  on  $\cS_{n,d}$ defined in \cite[\S4.3]{Pan-eta-SpO} for ortho-symplectic dual pairs and in \cite[\S3.3]{PanUU} for unitary dual pairs. 
The theta correspondence defines a relation between the sets $\cS_{n,d}\subset \cS_G$ and  $\cS_{n',d'}\subset \cS_{G'}$ for certain $d,d'$, and $\overline{\Theta}\colon\Irr^\unip(G)\longrightarrow \Irr^\unip(G')$ 
is defined inductively to be the smallest element in the set of elements of maximal order in
\[\Theta_{G'}(\pi_\Lambda)\backslash\left\{\overline{\Theta}_{G'}(\pi_{\Lambda_1})\,:\,\Lambda_1<\Lambda\right\}.\]

Thanks to Theorems~\ref{thm:reducUU}, \ref{thm:reducSpOeven} and \ref{thm:reducSpOodd}, which establish the compatibility of the theta correspondence with the bijections $\tcL_G$ and $\tcL_{G'}$, we can extend 
$\underline{\Theta}$ and $\overline{\Theta}$ outside the unipotent representations.

\begin{thm}
For every dual reductive dual pair  $(G,G')=(\rO_m,\Sp_{2n'})$ such that $m\le n'$, we have 
\[\underline{\Theta}_{G'}(\pi)=\overline{\Theta}_{G'}(\pi)=\eta(\pi),\]
for any $\pi\in\Irr(G)$, that is, both $\underline{\Theta}_{G'}$ and $\overline{\Theta}_{G'}$ are extensions of the eta 
correspondence from the stable stable range to arbitrary reductive dual pairs of Type I.
\end{thm}
\begin{proof}
See \cite[Theorem~6.13]{Pan-eta-SpO}.
\end{proof}

\begin{defn} {\rm 
A subrelation $\widetilde{\Theta}$ of $\Theta$ is \textit{semi-persistent} (on unipotent representations) if it satisfies the following conditions:
\begin{enumerate}
\item if either 
\begin{enumerate} 
\item $(G,G')=(\Sp_{2n}(k),\rO_{2n'}^+(k))$ and $\Lambda\in\cS_G$ with $\defect(\Lambda)=4d+1$; or
\item $(G,G')=(\rO_{2n'}^+(k),\Sp_{2n}(k))$ and $\Lambda\in\cS_G$ with $\defect(\Lambda)=4d$,
\end{enumerate}
then $\widetilde\Theta_{G'}(\pi_\Lambda)\ne\emptyset$ for any $n'\ge n-2d$;
\item if either 
\begin{enumerate} 
\item $(G,G')=(\Sp_{2n}(k),\rO_{2n'}^-(k))$ and $\Lambda\in\cS_G$ with $\defect(\Lambda)=4d+1$; or
\item $(G,G')=(\rO_{2n'}^-(k),\Sp_{2n}(k))$ and $\Lambda\in\cS_G$ with $\defect(\Lambda)=4d+2$,
\end{enumerate}
then $\widetilde\Theta_{G'}(\pi_\Lambda)\ne\emptyset$ for any $n'\ge n+2d+1$;
\item  $(G,G')=(\rU_n(k),\rU_{n'}(k)$ and if either
\begin{enumerate} 
\item  $n + n'+ \ell(\lambda_\infty)$ is even and $n'\ge n - \ell(\lambda_\infty)$; or
\item  $n + n' +\ell(\lambda_\infty)$ is odd and $n'\ge  n + \ell(\lambda_\infty)+1$,
\end{enumerate}
then $\widetilde\Theta_{G'}(\pi_\lambda)\ne\emptyset$.
\end{enumerate}
}
\end{defn}

\begin{defn} \

{\rm 
\begin{enumerate}
\item 
A subrelation $\widetilde{\Theta}$ of $\Theta$ is \textit{symmetric} if for each reductive dual pair $(G,G')$, any $(\pi,\pi')\in\Irr(G)\times\Irr(G')$,
we have $\pi'\in \widetilde\Theta_{G'}(\pi)$ if and only if $\pi\in \widetilde\Theta_{G}(\pi')$.
\item
A symmetric semi-persistent subrelation of $\Theta$ which is compatible with the bijections $\tcL_G$ and $\tcL_{G'}$ defined in \eqref{eqn:JordanUU} and \eqref{eqn:JordanSpO} is a \textit{theta-relation}.
\end{enumerate}
}
\end{defn}

\begin{thm}
Any theta-relation which properly contains $\underline\Theta$ or $\overline\Theta$ is not one-to-one. 

In other words, both $\underline\Theta$ and $\overline\Theta$ are maximal one-to-one theta-relations.
\end{thm}
\begin{proof}
See \cite[Corollary~7.3]{Pan-eta-SpO} and \cite[Proposition~4.7]{PanUU}.
\end{proof}

\section{A holographic duality} \label{sec:holo}
Classical error-correcting codes - used to transmit messages over noisy channels – have long been known to have rich connections to other branches of mathematics, including invariant theory, sphere packings, and modular forms.

Narain conformal field theories (CFTs) are a class of two-dimensional CFTs describing the geometry of the spacetime in string theory (see \cite{Narain1, Narain2}). They are characterized by a set of vertex operators whose left- and right-moving momenta span so-called Narain lattices. Narain CFTs are bosonic, non-chiral, and modular invariant when the lattices are even, Lorentzian, and self-dual, respectively. While they are simple theories of free compact bosons specified by the
lattices, they have a rather large continuous moduli space and exhibit a rich structure such as symmetry enhancement and dualities. Recently, they have received renewed interest  due to the revelation of their relation to quantum error-correcting codes (QECs). The CFTs whose Narain lattices are the even self-dual Lorentzian lattices built from qubit stabilizer codes are named \textit{Narain code CFTs}.

In \cite{DHM}, A. Dymarsky, J. Henriksson and B. McPeak studied  the holographic correspondence between 3d ``Chern–Simons gravity” and an \textit{ensemble} of 2d Narain code CFTs, showing that the mathematical identity underlying this holographic duality can be understood and rigorously proven using the framework of the theta correspondence over finite fields. The physical properties of the CFT  ensemble are linked to concepts from quantum information theory, specifically quantum stabilizer codes and classical even self-dual error-correcting codes. This rephrasing of holographic duality in terms of quantum codes provides new tools and perspectives for studying quantum gravity and how spacetime geometry might emerge from quantum information.

\end{document}